\documentclass[10pt,journal,compsoc]{IEEEtran}

\usepackage{amssymb}
\usepackage{amsmath}
\usepackage{amsfonts}
\usepackage{mathdots}

\usepackage[sort]{cite}
\usepackage{psfrag}
\usepackage{epsfig}
\usepackage{amsmath,bbm,times,stmaryrd}

\usepackage{amsthm}

 \usepackage[mathscr]{eucal}

\usepackage{tabularx}
\usepackage{multirow}
\usepackage{enumerate}

\usepackage[normal]{subfigure}

\usepackage{colortbl}
\usepackage{dcolumn}
\newcolumntype{.}{D{.}{.}{1.3}}

\usepackage{txfonts}
\usepackage[subnum]{cases}

\usepackage{equationarray}

\newcommand\dashrule{\leavevmode\xleaders\hbox{-}\hfill\kern0pt}

\def\argmin{\operatornamewithlimits{arg\,min}}
\def\diag{\operatorname{diag}}

\newcommand{\bn}{\bvec{n}}

\newcommand{\bx}{\bvec{x}}
\newcommand{\by}{\bvec{y}}

\newcommand{\bA}{{\bf A}}
\newcommand{\bB}{{\bf B}}

\newcommand{\bI}{{\bf I}}

\newcommand{\bL}{{\bf L}}

\newcommand{\bQ}{{\bf Q}}
\newcommand{\bR}{{\bf R}}

\newcommand{\bU}{{\bf U}}
\newcommand{\bV}{{\bf V}}

\newcommand{\bX}{{\bf X}}
\newcommand{\bY}{{\bf Y}}

\newcommand{\calO}{{\mathcal{O}}}

\newcommand{\blambda}{\mbox{\boldmath $\lambda$}}

\newcommand{\bsigma}{\mbox{\boldmath $\sigma$}}

\newcommand{\bSigma}{\mbox{\boldmath $\Sigma$}}

\newcommand{\bPhi}{\mbox{\boldmath $\Phi$}}
\newcommand{\bPsi}{\mbox{\boldmath $\Psi$}}

\newcommand{\be}{\begin{eqnarray}}
\newcommand{\ee}{\end{eqnarray}}

\newcommand{\matrixb}{\left[ \begin{array}}
\newcommand{\matrixe}{\end{array} \right]}

\newcommand{\sign}{\mathop{\rm sign}\nolimits}
\newcommand{\tr}{\mathop{\rm tr}\nolimits}

\def\*{\circledast}

\newtheorem{definition}{Definition}
\newtheorem{remark}{Remark}
\newtheorem{lemma}{Lemma}

\newcommand{\bvec}[1]{\boldsymbol{#1}}

\newcommand{\ve}{\bvec{e}}

\newcommand{\vm}{\bvec{m}}

\def\vectorize{\operatorname{vec}}
\newcommand{\vtr}[1]{\vectorize\hspace{-.3ex}\left(#1\right)}

\def\reshapename{\operatorname{reshape}}
\newcommand{\reshape}[2]{\reshapename\hspace{-.3ex}\left(#1,#2\right)}

\newcommand{\tensor}[1]{\boldsymbol{\mathscr{\MakeUppercase{#1}}}} 
\newcommand{\tA}{\tensor{A}}
\newcommand{\tB}{\tensor{B}}
\newcommand{\tC}{\tensor{C}}

\newcommand{\tE}{\tensor{E}}

\newcommand{\tL}{\tensor{L}}

\newcommand{\tR}{\tensor{R}}

\newcommand{\tT}{\tensor{T}}

\newcommand{\tX}{\tensor{X}}
\newcommand{\tY}{\tensor{Y}}
\newcommand{\tZ}{\tensor{Z}}

\newcommand{\minitab}[2][l]{\begin{tabular}{@{}#1}#2\end{tabular}}
\usepackage[vlined,ruled,commentsnumbered]{algorithm2e}

\usepackage{aurical}

\usepackage{arydshln}

\usepackage{url}
\usepackage{enumitem}
\usepackage{mathtools}

\usepackage{relsize}
\newcommand{\ltimesx}{\mathlarger{\ltimes}}
\newcommand{\rtimesx}{\mathlarger{\rtimes}}

\def\comment#1{}

\title{Tensor Networks for Latent Variable Analysis. Part I: Algorithms for Tensor Train Decomposition}
\author{\vspace{-.3ex}Anh-Huy Phan$^{*}$, Andrzej Cichocki, Andr{\'e} Uschmajew, Petr Tichavsk{\'y}, George Luta and Danilo Mandic
\thanks{A.-H. Phan and A. Cichocki are with the Lab for Advanced Brain Signal Processing, Brain Science Institute, RIKEN, Wakoshi, Japan, e-mail: (phan,cia)@brain.riken.jp.}
\thanks{A. Cichocki is also with Systems Research Institute PAS, Warsaw, Poland.}
\thanks{A. Uschmajew is with Hausdorff Center for Mathematics \& Institute for Numerical Simulation, University of Bonn, Germany, email: uschmajew@ins.uni-bonn.de.}
\thanks{P.  Tichavsk{\'y} is with Institute of Information Theory and Automation, Prague, Czech Republic, email: tichavsk@utia.cas.cz.}
\thanks{G. Luta is with Lombardi Comprehensive Cancer Center, Georgetown University, Washington D.C, USA, e-mail: george.luta@georgetown.edu}
\thanks{D. Mandic is with Imperial College, London, United Kingdom, email: d.mandic@imperial.ac.uk.}
}

\renewcommand{\CommentSty}[1]{\fontsize{8}{9}\selectfont\textnormal{\texttt{#1}}\unskip}
\SetKwComment{mtcc}{\% }{}
\SetKwSwitch{Switch}{Case}{Other}{switch}{}{case}{otherwise}{end switch}%

\newcounter{example} 

\newenvironment{example}
{\refstepcounter{example}\vspace{10pt}\par\noindent 
\textbf{Example \theexample. }
}
{}%

\makeatletter
\newsavebox{\@brx}
\newcommand{\llangle}[1][]{\savebox{\@brx}{\(\m@th{#1\langle}\)}%
  \mathopen{\copy\@brx\kern-0.5\wd\@brx\usebox{\@brx}}}
\newcommand{\rrangle}[1][]{\savebox{\@brx}{\(\m@th{#1\rangle}\)}%
  \mathclose{\copy\@brx\kern-0.5\wd\@brx\usebox{\@brx}}}
\makeatother

\begin{document}

\IEEEtitleabstractindextext{%
\begin{abstract}
Decompositions of tensors into factor matrices, which interact through a core tensor, have found numerous applications in signal processing and machine learning. A more general tensor model which represents data  as an ordered network of sub-tensors of order-2 or order-3 has, so far, not been widely considered in these fields, although this so-called tensor network decomposition has been long studied in quantum physics and scientific computing.
In this study, we present novel algorithms and applications of tensor network decompositions, with a particular focus on the tensor train decomposition and its variants. The novel algorithms developed for the tensor train decomposition update, in an alternating way, one or several core tensors at each iteration, and exhibit enhanced mathematical tractability and scalability to exceedingly large-scale data tensors. The proposed algorithms are tested in classic paradigms of blind source separation from a single mixture, denoising, and feature extraction, and achieve superior performance over the widely used truncated algorithms for tensor train decomposition. 
\end{abstract}

\begin{IEEEkeywords}
Tensor network, tensor train decomposition, Tucker-2 decomposition, truncated SVD, blind source separation from single mixture, image denoising, tensorization
\end{IEEEkeywords}}

\maketitle

\IEEEraisesectionheading{\section{Introduction}\label{sec:introduction}}

%

\IEEEPARstart{T}{ensor} decompositions (TDs) are rapidly finding application in signal processing paradigms, including the identification of independent components in multivariate data through the decomposition of higher order cumulant tensors,
signals retrieval in CDMA telecommunications, extraction of hidden components from neural data, training of dictionaries in supervised learning systems, image completion and various tracking scenarios. Most current applications are based on the CANDECOMP/PARAFAC (CPD)\cite{Harshman,Carroll_Chang} and the Tucker decomposition \cite{Tucker66,Lathauwer_HOOI}, while their variants, such as the PARALIND, PARATUCK \cite{CEM:CEM1206,Favier-deAlmeida14} or the Block term decomposition \cite{Lath-BCM}, the tensor deflation or tensor rank splitting \cite{Phan_tensordeflation_alg,Phan_ranksplitting_full} were developed with a specific task in mind; for a review see \cite{Kolda08,NMF-book,Cich-Lathp} and references therein.

Within tensor decompositions the data tensor is factorized into a set of factor matrices and a core tensor or a diagonal tensor, the entries of which model interaction between factor matrices. Such tensor decompositions are natural extensions of matrix factorizations, which  allows for most two-way factor analysis methods to be generalised to their multiway analysis counterparts. However, despite of mathematical elegance, such tensor decompositions easily become computationally intractable, or ill conditioned representations, particularly in CPD.

To help resolve these issues, which are a critical obstacle in a more widespread use of tensor decompositions in practical applications, we here consider another kind of tensor approximation, whereby multiple small core tensors are interconnected and construct an ordered network of such core tensors. 
More specifically, we focus on the Tensor Train (TT) decomposition, in which core tensors connect to only one or two other cores (see illustration in Fig.~\ref{fig_TTtensor}), so that, the tensor network (TN) acts as a ``train'' of tensors \cite{OseledetsTT09}. The TT decomposition has been brought into the tensor decomposition community through the work of Oseledets and Tyrtyshnikov \cite{OseledetsTT09}, although the model itself was developed earlier in quantum computation and chemistry under the name of the matrix product states (MPS) \cite{Klumper91,Vidal03}.
Compared to rank issues in standard tensor decompositions, the quasi-ranks in the TT decomposition can be determined in a stable way, e.g., through a rank-reduction using the truncated singular value decomposition. Moreover, by casting the data into the TT format, the paradigms of solving a huge system of linear equations, or eigenvalue decomposition of large-scale data can be reduced to solving  smaller scale sub-problems of the same kind\cite{Holtz-TT-2012,KressnerEIG2014}. 
Owing to the enhanced tractability in computation, the Hierarchical Tucker format and TTs have also been successfully used for tensor completion in e.g., seismic data analysis, hyperspectral imaging and parametric PDEs\cite{DaSilva2015131,Kressner2014,Rauhut2015,journals/siamsc/GrasedyckKK15},
Despite such success, TT decomposition as well as other tensor networks are yet to gain the same popularity in signal processing and machine learning as the standard CPD and Tucker decompositions. 
%
To this end, this article and its sequel aim to address this void in the literature, and present, for the first time, applications of tensor networks in some standard signal processing and machine learning paradigms, such as latent component analysis, denoising and feature extraction. We show that the framework presented can serve for the separation of signals even from a single data channel. We also present a novel tensor network based method to estimate factor matrices within CPD of high order tensors. 

As with many other tensor decompositions, the basic problem in the TT decomposition is to find optimal representation ranks of a tensor. Two different tasks may arise (i) when the TT-ranks are given, or (ii) when the approximation error is constrained to be smaller than a  predefined tolerance value or a predicted noise level.  Existing algorithms for the TT decomposition are based on truncated SVD and sequential projection
\cite{Vidal03,OseledetsTT09,OseledetsTT11}, whereby the core tensors are derived from leading singular vectors of the projected data onto the subspace of the other core tensors. This method is simple and works efficiently when data is amenable to  the so imposed strict models, as is the case in quantum physics. However, for general data, the ranks are not known beforehand, the truncation algorithm is less efficient, and the TT solutions do not achieve the optimal approximation error. On the other hand, for decompositions with a prescribed approximation accuracy, the algorithm is not guarantee to yield a tensor with minimal TT-rank. In this paper, we introduce novel algorithms to approximate a large-scale tensor by smaller-scale TT tensors with a particular emphasis on stability and minimum rank issues. This is achieved based on an alternating update scheme which sequentially updates one, two or three core tensors at a time.

The paper is organised as follows. The TT-tensors and operators for tensor manipulation are introduced in Section~\ref{sec::Preliminaries}. 
The TT-SVD algorithm is elaborated in Section~\ref{sec:ttsvd}. Since the TT-decomposition of order-3 tensors is equivalent to the Tucker-2 decomposition, algorithms for this case are presented in Section~\ref{sec:tt_order3}
and are used as a basic tool for higher order tensors. Section~\ref{sec:amcu} presents algorithms for the cases when the TT-rank is specified or when the noise level is given. 
We show that the decompositions considered can perform even faster when a data tensor is replaced by its crude TT-approximation. The algorithm for this case is presented in Section~\ref{sec:prior_compression}.
The proposed suite of algorithms for TT-decomposition is verified by simulations on signal and image de-noising and latent variable analysis. 
A new tensorization method is also proposed in the context of image denoising.

\section{Preliminaries}\label{sec::Preliminaries}

We shall next present the definitions of tensor contraction, tensor train decomposition, and orthogonalisation for a tensor train. 
The following three tensor contractions are defined for an order-$N$ tensor, $\tA$, of size $I_1 \times I_2 \times \cdots \times I_N$ and an order-$K$ tensor, $\tB$, of size  $J_1 \times J_2 \times \cdots \times J_K$.

\begin{definition}[Tensor train contraction]\label{def_train_contract}\label{def_boxtime}
The train contraction performs a tensor contraction between the last mode of tensor $\tA$ and the first mode of tensor $\tB$, where $I_N = J_1$,
to yield a tensor $\tC = \tA \bullet \tB$ of size $I_1  \times \cdots \times I_{N-1} \times J_2 \times \cdots \times J_K$, the elements of which are given by \\[-1.6em]
\be
c_{i_1,\ldots,i_{N-1},j_2, \ldots,j_K} = \sum_{i_N = 1}^{I_N}  a_{i_1,\ldots,i_{N-1},i_N} \, b_{i_N,j_2, \ldots,j_K}  \notag .
\ee
\end{definition}\vspace{-.8ex}
Fig.~\ref{fig_TTtensor} illustrates the principle of the train contraction.

\begin{definition}[Left contraction]\label{def_left_contract}
With $I_k = J_k$ for $k = 1, \ldots, n$, 
the $n$-modes left-contraction between $\tA$ and $\tB$, denoted by 
$
\tC = \tA  \ltimesx_n \tB    
$, computes a contraction product between their first $n$ modes,  
and yields a tensor $\tC$ of size $I_{n+1}  \times \cdots \times I_{N} \times J_{n+1} \times \cdots \times J_K$, defined as\\[-1.6em]
\be
c_{i_{n+1},\ldots,i_{N},j_{n+1}, \ldots,j_K} = \sum_{i_1 = 1}^{I_1}\cdots \sum_{i_n = 1}^{I_n}   a_{i_1, \ldots, i_n,i_{n+1},\ldots,i_{N}} \, b_{i_1, \ldots, i_n, j_{n+1},\ldots, j_K} \, . \notag
\ee
\end{definition}

\vspace{-.6em}
\begin{definition}[Right contraction]\label{def_right_contract}
With $I_{N-k} = J_{K-k}$ for $k = 0, 1, \ldots, n-1$,
the $n$-modes right tensor contraction between $\tA$ and $\tB$, denoted by $
\tC = \tA  \rtimesx_n \tB$, computes a contraction product between their last $n$ modes,  and yields a tensor $\tC$ of size $I_{1}  \times \cdots \times I_{N-n} \times J_{1} \times \cdots \times J_{K-n}$, defined as\\[-1.6em]
\be
c_{i_{1},\ldots,i_{N-n},j_{1}, \ldots,j_{K-n}} = \sum_{i_{N-n+1} = 1}^{I_{N-n+1}}\cdots \sum_{i_N = 1}^{I_N}   a_{i_1, \ldots,i_{N}} \, b_{j_1, \ldots, j_{K-n}, i_{N-n+1},\ldots, i_N} \, . \notag 
\ee
\end{definition}\vspace{-1ex}
 Fig.~\ref{fig_tt_tensor_et_contraction} illustrates the principles of the left and right contractions.

\setlength{\textfloatsep}{2ex}
\begin{figure}
\centering
\subfigure[A TT-tensor of rank-($R_1, R_2,\ldots, R_{N-1}$)]
{
\psfrag{X1}[c][c]{\scalebox{1}{\color[rgb]{0,0,0}\setlength{\tabcolsep}{0pt}\begin{tabular}{c}  \smaller\smaller$\tX_1$\end{tabular}}}%
\psfrag{X2}[c][c]{\scalebox{1}{\color[rgb]{0,0,0}\setlength{\tabcolsep}{0pt}\begin{tabular}{c}  \smaller\smaller$\tX_2$\end{tabular}}}%
\psfrag{Xn1}[cc][cc]{\scalebox{1}{\color[rgb]{0,0,0}\setlength{\tabcolsep}{0pt}\begin{tabular}{c} \smaller\smaller \hspace{1ex}$\tX_{N-{\color[rgb]{.5,.5,.5}1}}$\end{tabular}}}%
%
\psfrag{XN}[c][c]{\scalebox{1}{\color[rgb]{0,0,0}\setlength{\tabcolsep}{0pt}\begin{tabular}{c}  \smaller\smaller$\tX_{N}$\end{tabular}}}%
\psfrag{I1}[c][c]{\scalebox{1}{\color[rgb]{0,0,0}\setlength{\tabcolsep}{0pt}\begin{tabular}{c}  \smaller$I_1$\end{tabular}}}%
\psfrag{I2}[c][c]{\scalebox{1}{\color[rgb]{0,0,0}\setlength{\tabcolsep}{0pt}\begin{tabular}{c}  \smaller$I_2$\end{tabular}}}%
\psfrag{In1}[c][c]{\scalebox{1}{\color[rgb]{0,0,0}\setlength{\tabcolsep}{0pt}\begin{tabular}{c}  \smaller$I_{N-1}$\end{tabular}}}%
\psfrag{IN}[c][c]{\scalebox{1}{\color[rgb]{0,0,0}\setlength{\tabcolsep}{0pt}\begin{tabular}{c}  \smaller$I_{N}$\end{tabular}}}%
\psfrag{R1}[c][c]{\scalebox{1}{\color[rgb]{0,0,0}\setlength{\tabcolsep}{0pt}\begin{tabular}{c}  \smaller $R_1$\end{tabular}}}%
\psfrag{R2}[c][c]{\scalebox{1}{\color[rgb]{0,0,0}\setlength{\tabcolsep}{0pt}\begin{tabular}{c}  \smaller $R_2$\end{tabular}}}%
\psfrag{Rn2}[c][c]{\scalebox{1}{\color[rgb]{0,0,0}\setlength{\tabcolsep}{0pt}\begin{tabular}{c}  \smaller $R_{N-2}$\end{tabular}}}%
\psfrag{Rn1}[c][c]{\scalebox{1}{\color[rgb]{0,0,0}\setlength{\tabcolsep}{0pt}\begin{tabular}{c}  \smaller $R_{N-1}$\end{tabular}}}%
\psfrag{...}[c][c]{\scalebox{1}{\color[rgb]{0,0,0}\setlength{\tabcolsep}{0pt}\begin{tabular}{c}  \large $\cdots$\end{tabular}}}%
\psfrag{train contraction}[c][c]{\scalebox{1}{\color[rgb]{0,0,0}\setlength{\tabcolsep}{0pt}\begin{tabular}{c}\smaller train contraction \end{tabular}}}%
\includegraphics[width=.55\linewidth, trim = 0.0cm 0cm 0cm 0cm,clip=true]{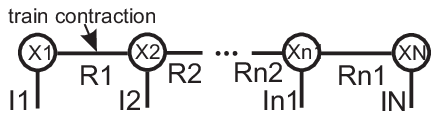}\label{fig_TTtensor}
}
\psfrag{I1}[c][c]{\scalebox{1}{\color[rgb]{0,0,0}\setlength{\tabcolsep}{0pt}\begin{tabular}{c}  \smaller $I_1$\end{tabular}}}%
\psfrag{In}[c][c]{\scalebox{1}{\color[rgb]{0,0,0}\setlength{\tabcolsep}{0pt}\begin{tabular}{c}  \smaller  $I_n$\end{tabular}}}%
\psfrag{IN}[c][c]{\scalebox{1}{\color[rgb]{0,0,0}\setlength{\tabcolsep}{0pt}\begin{tabular}{c} \smaller  $I_N$\end{tabular}}}%
\psfrag{J1}[c][c]{\scalebox{1}{\color[rgb]{0,0,0}\setlength{\tabcolsep}{0pt}\begin{tabular}{c}  \smaller  $J_1$\end{tabular}}}%
\psfrag{Jn}[c][c]{\scalebox{1}{\color[rgb]{0,0,0}\setlength{\tabcolsep}{0pt}\begin{tabular}{c}  \smaller  $J_n$\end{tabular}}}%
\psfrag{JN}[c][c]{\scalebox{1}{\color[rgb]{0,0,0}\setlength{\tabcolsep}{0pt}\begin{tabular}{c}  \smaller  $J_N$\end{tabular}}}%
\hfill
\subfigure[Left and right contractions]{
\includegraphics[width=.18\linewidth, trim = 0cm .01cm -0.3cm 0cm,clip=true]{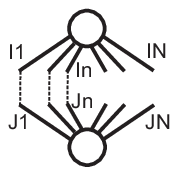}
\label{fig_leftcontraction}
\hfill
\psfrag{Jm}[c][c]{\scalebox{1}{\color[rgb]{0,0,0}\setlength{\tabcolsep}{0pt}\begin{tabular}{c}  \smaller  $J_m$\end{tabular}}}%
\psfrag{Im}[c][c]{\scalebox{1}{\color[rgb]{0,0,0}\setlength{\tabcolsep}{0pt}\begin{tabular}{c}  \smaller  $I_m$\end{tabular}}}%
\includegraphics[width=.18\linewidth, trim = -.3cm .01cm 0cm 0cm,clip=true]{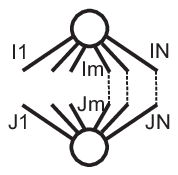}\label{fig_rightcontraction}
\label{fig_tt_tensor_et_contraction}}
%
\caption{\subref{fig_TTtensor} Graphical illustration of a TT-tensor $\tX = \tX_1 \bullet \tX_2 \bullet \cdots  \bullet \tX_N$.  A node represents a 3-rd order core tensor $\tX_n$ of size $R_{n-1}\times I_n \times R_n$. \subref{fig_rightcontraction} Left and right contractions between two tensors.}
\label{fig_TTtensor_}
\end{figure}


\begin{definition}[Tensor train decomposition\cite{Vidal03,OseledetsTT09}]\label{def_tttensor}
A tensor train decomposition of a tensor $\tX$ of size $I_1 \times I_2 \times \cdots \times I_N$, with a TT-rank $(R_1, R_2, \ldots,R_{N-1})$, has the form
\begin{align}
\tX = \sum_{r_1 = 1}^{R_1} \, \sum_{r_2 = 1}^{R_2}  \cdots \sum_{r_{N-1} = 1}^{R_{N-1}} \tX_1(:,r_1,1) \circ \tX_2(r_1,:,r_2) \circ \cdots \circ \tX_N(r_{N-1},:,1) 
\notag 
\end{align}
where $\tX_n$ are core tensors of size $R_{n-1} \times I_n \times R_{n}$, $R_0 = R_{N} = 1$, 
and $\tX_n(r_{n-1},:,r_{n})$ are vertical fibers of $\tX_n$, while 
the symbol $``\circ''$ designates the outer product.
\end{definition}
A tensor $\tX$ in the TT-format is called a TT-tensor, and can be expressed equivalently through: 
\begin{itemize}[leftmargin=*,align=left]
\item Train-contractions  as 
$
	\tX =  \tX_1 \bullet \tX_2 \bullet \cdots \bullet \tX_{N-1} \bullet \tX_N \, , \notag 
$
\item 
A product of its sub TT-tensors
$
\tX = \tX_{<n} \bullet \tX_{n:m} \bullet  \tX_{>m}  
$, where $\tX_{<n}$ and $\tX_{>m}$ are respectively the TT-tensors composed by all core tensors to the left of $\tX_n$
and to the right of $\tX_m$, $m\ge n$, that is \\[-1.8em]
\be
\tX_{<n} &=& \tX_1 \bullet \tX_2 \bullet \cdots \bullet \tX_{n-1} \,\notag \label{equ_Xlen} \\[-.3em]
\tX_{>m} &=& \tX_{m+1} \bullet \tX_{m+2} \bullet \cdots \bullet \tX_{N}\notag \label{equ_Xgen}\\[-.3em]
\tX_{n:m} &=& \tX_{n} \bullet \tX_{n+1} \bullet \cdots \bullet \tX_{m}\notag \label{equ_Xnm}.
\ee
\end{itemize}
It is important to note that a TT-representation can always be compressed, e.g., using the TT-SVD algorithm (Algorithm~\ref{alg_TT_svd} below) with perfect accuracy $\epsilon$ = 0,
 such that the representation ranks satisfy\\[-1.3em]
\begin{equation}
R_n \le \min (R_{n-1}I_n, I_{n+1}R_{n+1})  \quad \text{or}\;\;
	R_n \le \min\left(\prod_{k = 1}^{n} I_k, \prod_{l = n+1}^{N} I_l  \right), \notag 
\end{equation}\\[-1.2em]
 for $n = 1, 2, \ldots, N-1$. The first inequalities above imply the second ones.

\begin{definition}[\bf Tensor unfolding]
Let $\bn_1, \bn_2,\ldots, \bn_J$ be disjoint subsets, and $[\bn_1, \bn_2, \ldots, \bn_J]$  a permutation of $[1, 2, \ldots, N]$, where $\bn_j = [n_{j}(1), n_{j}(2), \ldots, n_{j}(K_j)]$ and $K_1 + K_2 + \cdots +K_J = N$.
The mode-($\bn_1,\bn_2,\ldots, \bn_J$) unfolding converts an order-$N$ tensor $\tX$ into an order-$J$ tensor $\tY$, given by $
\tX(i_1,i_2, \ldots, i_N) = \tY(i_{\bn_1}, i_{\bn_2}, \ldots, i_{\bn_J})$,
where $i_{\bn_j}$ is a linear index of ($i_{n_{j}(1)}, i_{n_{j}(2)}, \ldots, i_{n_{j}(K_j)}$) \cite{Phan2012-Kron}.

The unfolding operator is denoted by 
$\tY = [\tX]_{(\bn_1,\bn_2,\ldots, \bn_J)}$.
When $\vm = \{1,\ldots,N\}\setminus {\bn}$, and its entries are sorted in an ascending order, the mode-$(\bn,\vm)$ unfolding is also known as mode-$\bn$ matricization and is denoted by $ [\tX]_{(\bn,\vm)} = \bX_{(\bn)}$. 
\end{definition}

\begin{definition}[\bf Left and right orthogonality conditions for the core tensor $\tX_n$\cite{Holtz-TT-2012,Kressner2014MC}]
Consider a TT-tensor $\tX = \tX_1 \bullet \tX_2 \bullet \cdots \bullet \tX_N$. Then, its core tensor $\tX_n$ is said to satisfy the left-orthogonality condition if 
$
	{\tX}_n \ltimesx_{2} {\tX}_n = \bI_{R_{n}} $, 
and the right-orthogonality condition if
$  {\tX}_n \rtimesx_{2}  {\tX}_n  = \bI_{R_{n-1}} $. 
\end{definition}
The mode-$n$ left orthogonalisation can be achieved using the orthogonal Tucker-1 decomposition of $\tX_n$ in the form $\tX_n = {\tilde{\tX}}_{n} \bullet  \bL$, or from the QR decomposition of the mode-(1,2) matricization $[\tX_{n}]_{(1,2)} = \bQ \, \bR$, where $[{\tilde{\tX}}_{n}]_{(1,2)}  = \bQ$ is an orthogonal matrix,  and $\bL  = \bR$. The mode-$n$ left orthogonalised TT-tensor $\tX$ now becomes
\begin{align}
\tX 
&= \tX_1 \bullet  \cdots \bullet  \tX_{n-1} \bullet {\tilde{\tX}}_{n} \bullet (\bL \bullet \tX_{n+1}) \bullet \cdots \bullet \tX_N  \notag \label{equ_leftorthogonal}.
\end{align}
Similarly, the mode-$n$ right orthogonalisation performs the orthogonal Tucker-1 decomposition $\tX_n  = \bR \bullet {\tilde{\tX}}_{n}$, and the resulting TT-tensor $\tX$ becomes\\[-1.5em]
\begin{align}
\tX 
&= \tX_1 \bullet  \cdots \bullet  (\tX_{n-1} \bullet  \bR) \bullet {\tilde{\tX}}_{n} \bullet   \tX_{n+1} \bullet \cdots \bullet \tX_N \, .\notag 
\end{align}

\begin{definition}[\bf Left orthogonalisation up to mode $n$]
The left orthogonalization of a tensor $\tX$ up to mode-$n$ performs $(n-1)$ left orthogonalizations of the core tensors to the left of $n$ such that 
${\tX}_k \ltimesx_{2} {\tX}_k = \bI_{R_{k}}$  for $k = 1, 2, \ldots, n-1$.
\end{definition}
\begin{definition}[\bf Right orthogonalisation up to mode $n$]
The right orthogonalization of a tensor  $\tX$ up to mode-$n$ performs $(N-n)$ right orthogonalizations of the core tensors to the right of $n$ such that 
${\tX}_k \rtimesx_{2} {\tX}_k = \bI_{R_{k-1}}$  for $k = n+1, n+2, \ldots, N$.
\end{definition}
In this paper, we consider the following two approximations of a tensor $\tY$ by a TT-tensor $\tX$
\begin{itemize}[leftmargin=*,align=left]
\item {\bf The TT-approximation with a given TT-rank}, which is based on a minimisation of the Frobenius norm of the approximation error, in the form
\be
	\min \quad D = \|\tY  - \tX\|_F^2 \, \label{equ_cost_TT}.
\ee
\item {\bf The TT-approximation with a given approximation accuracy}, which is typically used in the presence of noise or when the TT-rank is not specified, and is based on the solution of a denoising problem,  
\be
	\|\tY  - \tX\|_F^2  \le  \varepsilon^2\, \label{equ_cost_TT_noisecontraint},
\ee  
such that the TT-rank of $\tX$ is minimum. In (\ref{equ_cost_TT_noisecontraint}), $\varepsilon^2$ represents the noise level, or an approximation accuracy.
\end{itemize}

\section{TT-SVD or TT-truncation algorithm}\label{sec:ttsvd}

In many practical settings, the tensor train decomposition can be performed efficiently using a sequential projection and truncation algorithm, known as the TT-SVD\cite{Vidal03,OseledetsTT09,OseledetsTT11}. More specifically, the first core $\tX_1$ is obtained from the $R_1$ leading singular vectors of the reshaping matrix $\bY_{(1)}$, subject to the error norm being less than $\epsilon$ times the data norm, that is,
\be
	\|\bY_{(1)} - \bU \, \diag(\bsigma) \, \bV^T \|_F^2 \le  \epsilon^2 \, \|\bY_{(1)}\|_F^2  \notag \label{eq_tt_truncation}
\ee
or $\|\bsigma\|_2^2 \ge (1-\epsilon^2) \|\bY_{(1)}\|_F^2$.
The projected data $\diag(\bsigma) \bV^T$ is then reshaped into a matrix $\bY_2$ of size $(R_1 I_2) \times (I_3I_4\cdots I_N)$, and the second core tensor $\tX_2$ is estimated from the leading left singular vectors of this matrix, whereas the rank $R_2$ is chosen such that the norm of the residual is less than $\sqrt{1-\epsilon^2} \|\bY_2\|_F$.

The sequential projection and truncation procedure is repeated in order to find the remaining core tensors.
The algorithm, summarised in Algorithm~\ref{alg_TT_svd}, executes only $(N-1)$ sequential data projections and $(N-1)$ truncated-SVD of the projected data in order to estimate $N$ core tensors, and is quite simple to implement.
The TT-SVD algorithm can be modified for a TT-decomposition with TT-ranks specified, and can be efficiently implemented if the input is already provided in the TT format with small ranks and is used for further truncation.
This two-stage ``TT-SVD and truncation'' procedure is illustrated in Example~\ref{ex_TT_reconstruction_exp} in Section~\ref{sec::examples}.
In terms of the approximation accuracy, it can be shown that \cite{OseledetsTT09}
\be
	\|\tY - \tX\|_F^2 \le \sum_{k = 1}^{N-1} \epsilon_k^2		\notag 
\ee
where $\epsilon_k$ is the truncation error at $k$-th step.

When the data admits the TT format with small noise, the TT-SVD works well, however, the algorithm is less efficient when data is heavily corrupted by noise or when a TT-approximation is with low TT-rank. 
\begin{remark}More specifically, for the approximation problem in (\ref{equ_cost_TT}), TT-SVD is not guaranteed to achieve the minimum approximation error, as illustrated in Examples~\ref{ex_TT_reconstruction_exp} and \ref{ex_TT_separation_exp} in Section~\ref{sec::examples}.
\end{remark}\vspace{-.5em}
For the denoising problem in (\ref{equ_cost_TT_noisecontraint}), the resulting TT-tensor from TT-SVD satisfies the approximation condition, but often exhibits a relatively high TT-rank. 
\begin{remark}
An increase in the TT-rank of $\tX$ makes it easier to explain the data, so that the approximation error tends to be smaller than the tolerance error $\varepsilon^2$. However, when the TT-ranks are high, adding more terms into $\tX$ implies adding noise into the approximation, and reducing the reconstruction error. For the case of TT-SVD, this is illustrated in Example~\ref{ex_TT_denoising_signal_2}.
\end{remark}
In other words, TT-SVD tends to select a higher TT-rank than needed for the denoising problem. 
Following on the two Remarks above, the next sections present more efficient algorithms for the two approximation problems in (\ref{equ_cost_TT}) and (\ref{equ_cost_TT_noisecontraint}).

\setlength{\textfloatsep}{1ex}
\setlength{\algomargin}{.7em}
\begin{algorithm}[t!]
\SetFillComment
\SetSideCommentRight
\CommentSty{\footnotesize}
\caption{{\tt{TT-SVD\cite{Vidal03,OseledetsTT09}}}\label{alg_TT_svd}}
\DontPrintSemicolon \SetFillComment \SetSideCommentRight
\KwIn{Data tensor $\tY$:  $(I_1 \times I_2 \times \cdots \times I_N)$, TT-rank $(R_1,R_2,\ldots,R_{N-1})$ or approximation accuracy $\epsilon$ 
} 
\KwOut{A TT-tensor $\tX =  \tX_1 \bullet \tX_2 \bullet \cdots \bullet \tX_N$  
 such that $\min \|\tY - \tX\|_F^2$ or $\|\tY - \tX \|_F^2 \le \epsilon^2 \|\tY\|_F^2$
} \SetKwFunction{mreshape}{reshape}
\SetKwFunction{tsvd}{truncated\_svd} 
\SetKwFunction{bestTT}{bestTT\_approx} 
\Begin{
\For {$n = 1, \ldots, N-1$}{
\nl 	$\bY = \mreshape(\tY, (I_n\,R_{n-1}) \times \prod_{k=n+1}^{N} I_k)$\;
\nl 	Truncated SVD $\bY \approx  \bU \, \diag(\bsigma) \, \bV^T$ with given rank $R_{n}$ or  such that $\displaystyle  \|\bsigma\|_2^2  \ge (1-\epsilon^2) \|\bY \|_F^2$\;
\nl 	$\tX_n = \mreshape(\bU, R_{n-1} \times I_n \times R_{n})$\;
\nl 	$\tY \leftarrow \diag(\bsigma)\, \bV^T$\;
}
\nl $\tX_{N} = \tY$\;
}
\end{algorithm} 

\comment{
 \setlength{\algomargin}{.7em}
\begin{algorithm}[Ht!]
\SetFillComment
\SetSideCommentRight
\CommentSty{\footnotesize}
\caption{{\tt{Tucker-2 or Order-3 TT decomposition}}\label{alg_TT3}}
\DontPrintSemicolon \SetFillComment \SetSideCommentRight
\KwIn{Data tensor $\tY$:  $(I_1 \times I_2 \times I_3)$,  estimation accuracy $\varepsilon$
} 
\KwOut{TT-tensor $\tX = \bX_1 \bullet \tX_2 \bullet \bX_3$ of minimum rank-$ (R_1, R_{2})$ such that $\|\tY - \tX\|_F^2 \le \varepsilon^2$}
 \SetKwFunction{mreshape}{reshape}
\SetKwFunction{eig}{svds} 
\Begin{
\nl Initialize $\tX_3$ as an identity matrix of size $I_3 \times I_3$\;
\Repeat{a stopping criterion is met}{
\nl $\bQ_1 = (\tY \bullet \bX_3^T)  \rtimesx_2   (\tY \bullet \bX_3^T)$\;
\nl EVD of $\bQ_1 \approx \bX_1 \diag(\lambda_1, \ldots, \lambda_{R_1}) \, \bX_1^T $ such that $\sum_{r = 1}^{R_1}  \lambda_{r} \ge \|\tY\|_F^2 - \varepsilon^2 \ge \sum_{r = 1}^{R_1-1}  \lambda_{r} $\;
\nl $\bQ_3 = (\bX_1^T \bullet \tY)  \, \ltimesx_2   \,  (\bX_1^T \bullet \tY)$\;
\nl EVD of  $\bQ_3 \approx \bX_3^T \diag(\lambda_1, \ldots, \lambda_{R_2}) \, \bX_3 $ such that $\sum_{r = 1}^{R_2}  \lambda_{r}  \ge \|\tY\|_F^2 - \varepsilon^2 \ge \sum_{r = 1}^{R_2-1}  \lambda_{r}  $\; 
}
\nl $\tX_2 = \bX_1^T \bullet \tY \bullet \bX_3^T$
}
\end{algorithm} 
}

\section{A TT-decomposition for Order-3 Tensors}\label{sec:tt_order3}

Before presenting algorithms for the TT-decomposition of tensors of high order, we shall start with the TT-decomposition for order-3 tensors, and illuminate its relation to the Tucker-2 decomposition\cite{Tucker66,Lathauwer_HOOI}. The algorithm developed in this section will serve as a basis for updating core tensors in TT-decompositions of higher order tensors. 
%

\begin{definition}[Tucker-2 decomposition\cite{Tucker66}]\label{def_tucker2}
Tucker-2 decomposition of an order-3 tensor $\tY$ of size $I_1 \times I_2 \times I_3$ is given by\\[-1.4em]
\be
\tY  = \bX_1 \bullet \tX_2 \bullet \bX_3\,, \label{eq_tucker2}
\ee
where $\tX_2$ is the core tensor of size $R_1 \times I_2 \times R_2$, 
$\bX_1$ and $\bX_3$ are the two factor matrices of sizes $I_1 \times R_1$ and $R_2 \times I_3$, respectively, while the multilinear rank of the decomposition is $(R_1, R_2)$.
\end{definition}
By definition, the Tucker-2 decomposition is a TT decomposition of an order-3 tensor.
Because of rotational ambiguity, without loss in generality, the matrices $\bX_1$ and $\bX_3$ can be assumed to have orthonormal columns (for $\bX_1$) and rows (for $\bX_3$), that is $\bX_1^T \bX_1 = \bI_{R_1}$ and $\bX_3 \bX_3^T = \bI_{R_2}$. The second core tensor $\tX_2$ is then given in a closed-form as $\tX_2  = \bX_1^T \bullet \tY \bullet \bX_3^T$, 
and the Frobenius norm, $\|\bY - \bX_1 \bullet \tX_2 \bullet \bX_3\|_F^2$, of the approximation can be expressed as \\[-1.5em]
\be
D &=& \|\tY\|_F^2 - \| \bX_1^T \bullet \tY \bullet \bX_3^T \|_F^2 \notag\\
    &=& \|\tY\|_F^2 - \tr( \bX_1^T \, \bQ_1 \,  \bX_1)  \, ,	\notag 
\ee\\[-1.5em]
where $\bQ_1 = (\tY \bullet \bX_3^T)  \rtimesx_2   (\tY \bullet \bX_3^T)$ is a symmetric matrix of size $I_1 \times I_1$.

\vspace{-.5em}
\begin{remark}
{\bf For the TT-decomposition in (\ref{equ_cost_TT})}, the new estimate $\bX_1$ comprises $R_1$ principal eigenvectors of $\bQ_1$.
\end{remark}
\vspace{-.8em}
\begin{remark}
{\bf{For the denoising problem in (\ref{equ_cost_TT_noisecontraint})}}, $\bX_1$ is obtained as a solution to the following problem
\be
\tr(\bX_1^T  \, \bQ_1  \,  \bX_1)  \ge \|\tY\|_F^2 - \varepsilon^2.    \notag 
\ee
This implies that $\bX_1$ takes $R_1$ principal eigenvectors of $\bQ_1$, where $R_1$ is the smallest number of eigenvalues $\lambda_1 \ge \lambda_2 \ge \cdots \ge \lambda_{R_1}$ of $\bQ_1$ such that their norm exceeds the threshold $\|\tY\|_F^2 - \varepsilon^2$, 
that is
\be
\sum_{r = 1}^{R_1} \lambda_r \ge  \|\tY\|_F^2 - \varepsilon^2  > \sum_{r = 1}^{R_1-1} \lambda_r  \,. \notag \label{equ_determine_R2}
\ee 
Similarly, the core matrix $\bX_3$ of size $R_2 \times I_3$ comprises $R_2$ principal eigenvectors of the matrix
$\bQ_3 = (\bX_1^T \bullet \tY)  \, \ltimesx_2   \,  (\bX_1^T \bullet \tY)$, 
where $R_2$ is either given or determined based on the accuracy $\|\tY\|_F^2 - \varepsilon^2$.
The algorithm sequentially updates  $\bX_1$ and $\bX_3$.
\end{remark}

\section{Alternating Multi-Cores Update Algorithms}\label{sec:amcu}

This section presents novel algorithms for the TT-decomposition. 
We first present a simple form of the Frobenius norm of a TT-tensor, followed by a formulation of optimisation problems to update single or a few core tensors. 


\begin{lemma}[\bf Frobenius norm of a TT-tensor]\label{lem_Frob_norm}
Under the left-orthogonalisation up to $\tX_n$, and the right-orthogonalisation up to $\tX_m$, where $n\le m$, the Frobenius norm of a TT-tensor $\tX = \tX_1 \bullet \tX_2 \bullet \cdots \bullet \tX_N$ is  equivalent to the Frobenius norm of $\tX_{n:m}$, that is, $
\|\tX\|_F^2   = \|\tX_{n:m}\|_F^2 .	 \notag \label{eq_frob_Xnm}$
\end{lemma}
\begin{proof}
With the above left and right orthogonalisations, the two matricizations  $[\tX_{<n}]_{(n)}^T$ and 
$[\tX_{>m}]_{(1)}^T$ are orthogonal matrices. Hence, 
$
\| \tX \|_F^2  = \|  [\tX_{<n}]_{(n)}^T   \bullet \tX_{n:m} \bullet [\tX_{>m}]_{(1)} \|_F^2  
 =\|  \tX_{n:m}\|_F^2$.
\end{proof}

\setlength{\textfloatsep}{1ex}
\begin{figure}
\centering
\psfrag{Y}[c][c]{\scalebox{1}{\color[rgb]{0,0,0}\setlength{\tabcolsep}{0pt}\begin{tabular}{c}  \smaller  $\tY$\end{tabular}}}%
\psfrag{X1}[c][c]{\scalebox{1}{\color[rgb]{0,0,0}\setlength{\tabcolsep}{0pt}\begin{tabular}{c} \smaller   $\tX_1$\end{tabular}}}%
\psfrag{Xn}[c][c]{\scalebox{1}{\color[rgb]{0,0,0}\setlength{\tabcolsep}{0pt}\begin{tabular}{c}   \smaller $\tX_n$\end{tabular}}}%
\psfrag{Xm}[c][c]{\scalebox{1}{\color[rgb]{0,0,0}\setlength{\tabcolsep}{0pt}\begin{tabular}{c}  \smaller  $\tX_{m}$\end{tabular}}}%
\psfrag{Xm1}[c][c]{\scalebox{1}{\color[rgb]{0,0,0}\setlength{\tabcolsep}{0pt}\begin{tabular}{c}  \smaller  $\tX_{m+1}$\end{tabular}}}%
\psfrag{Xn1}[c][c]{\scalebox{1}{\color[rgb]{0,0,0}\setlength{\tabcolsep}{0pt}\begin{tabular}{c}  \smaller $\tX_{n-1}$\end{tabular}}}%
\psfrag{XN}[c][c]{\scalebox{1}{\color[rgb]{0,0,0}\setlength{\tabcolsep}{0pt}\begin{tabular}{c}   \smaller  $\tX_{N}$\end{tabular}}}%
\psfrag{I1}[c][c]{\scalebox{1}{\color[rgb]{0,0,0}\setlength{\tabcolsep}{0pt}\begin{tabular}{c}   \smaller $I_1$\end{tabular}}}%
\psfrag{In1}[c][c]{\scalebox{1}{\color[rgb]{0,0,0}\setlength{\tabcolsep}{0pt}\begin{tabular}{c}   \smaller $I_{n-1}$\end{tabular}}}%
\psfrag{Im1}[c][c]{\scalebox{1}{\color[rgb]{0,0,0}\setlength{\tabcolsep}{0pt}\begin{tabular}{c}  \smaller  $I_{m+1}$\end{tabular}}}%
\psfrag{In}[c][c]{\scalebox{1}{\color[rgb]{0,0,0}\setlength{\tabcolsep}{0pt}\begin{tabular}{c}  \smaller  $I_n$\end{tabular}}}%
\psfrag{Im}[c][c]{\scalebox{1}{\color[rgb]{0,0,0}\setlength{\tabcolsep}{0pt}\begin{tabular}{c}  \smaller  $I_{m}$\end{tabular}}}%
\psfrag{IN}[c][c]{\scalebox{1}{\color[rgb]{0,0,0}\setlength{\tabcolsep}{0pt}\begin{tabular}{c}  \smaller  $I_{N}$\end{tabular}}}%
\psfrag{R1}[c][c]{\scalebox{1}{\color[rgb]{0,0,0}\setlength{\tabcolsep}{0pt}\begin{tabular}{c}  \smaller $R_1$\end{tabular}}}%
\psfrag{R2}[c][c]{\scalebox{1}{\color[rgb]{0,0,0}\setlength{\tabcolsep}{0pt}\begin{tabular}{c}  \smaller $R_2$\end{tabular}}}%
\psfrag{Rn2}[c][c]{\scalebox{1}{\color[rgb]{0,0,0}\setlength{\tabcolsep}{0pt}\begin{tabular}{c}  \smaller $R_{N-2}$\end{tabular}}}%
\psfrag{Rn1}[c][c]{\scalebox{1}{\color[rgb]{0,0,0}\setlength{\tabcolsep}{0pt}\begin{tabular}{c}  \smaller $R_{N-1}$\end{tabular}}}%
\psfrag{...}[c][c]{\scalebox{1}{\color[rgb]{0,0,0}\setlength{\tabcolsep}{0pt}\begin{tabular}{c}  \smaller\smaller $\cdots$\end{tabular}}}%
\includegraphics[width=.7\linewidth, trim = 0.0cm 0cm 0cm 0cm,clip=true]{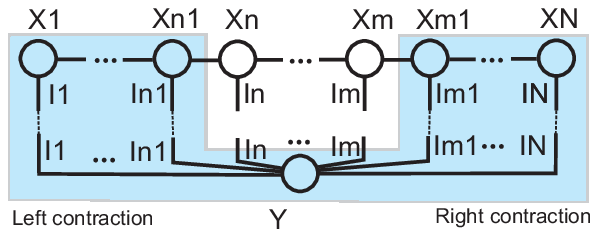}
\vspace{-.8ex}
\caption{Graphical illustration of a tensor contraction of a TT-tensor $\tX  = \tX_1 \bullet \tX_2 \bullet \cdots \bullet \tX_{N-1} \bullet \tX_N$ and a tensor $\tY$ over all modes but the modes-($n,n+1,\ldots,m$). 
The contraction is performed inside the shaded area, and yields a tensor $\tT_{n:m}$ of size $R_{n-1} \times I_n \times I_{n+1} \times \cdots \times I_m \times R_{m}$.}
\label{fig_Tnm}
\end{figure}

\vspace{-1em}
\subsection{The Objective Function and Generalized Framework for The Alternating  Multicore Update Algorithm}\label{sec:al_ncores}

%
%

We now proceed to simplify the two optimisation problems considered to those for sub TT-tensors which comprise a single core or a few consecutive core tensors.
For this purpose, we assume that the TT-tensor $\tX$ is left-orthogonalised up to $\tX_n$ and right-orthogonalized up to $\tX_m$, where $m$ can take one of the values $n, n+1$ or $n+2$.

Let $\tX_{n:m} = \tX_n \bullet \tX_{n+1} \bullet \cdots \bullet \tX_m$, then following Lemma~\ref{lem_Frob_norm}, the error cost function in (\ref{equ_cost_TT}) and in (\ref{equ_cost_TT_noisecontraint}) can be written as 
\be
D &=& \|\tY\|_F^2 +  \|\tX\|_F^2  - 2 \langle \tY, \tX\rangle \notag\\
&=&  \|\tY\|_F^2 + \|\tX_{n:m} \|_F^2  - 2  \langle \tT_{n:m}, \tX_{n:m}\rangle \notag\\
&=&  \|\tY\|_F^2 - \|\tT_{n:m} \|_F^2  + \| \tT_{n:m} - \tX_{n:m} \|_F^2    \label{eq_cost_Gnm}
\ee 
where $\tT_{n:m}$ is of size $R_{n-1} \times I_n \times \cdots \times I_m \times R_{m}$, and represents a tensor contraction between $\tY$ and $\tX$ along all modes but the modes-$(n,n+1,\ldots, m)$, i.e., left contraction along the first $(n-1)$-modes and right contraction along the last $(N-m)$-modes, expressed as
\be
\tT_{n:m} = (\tX_{<n} \, \ltimesx_{n-1} \tY )  \, \rtimesx_{N-m}   \, \tX_{>m}\,, \quad \text{for} \;\; n = 1, 2, \ldots\label{equ_compute_Tnm}
\ee
Fig.~\ref{fig_Tnm} illustrates the computation of the contracted tensor $\tT_{n:m}$.
The objective function in (\ref{eq_cost_Gnm}) indicates that the sub TT-tensor $\tX_{n:m}$ is the best approximation to $\tT_{n:m}$ in both problems (\ref{equ_cost_TT}) and (\ref{equ_cost_TT_noisecontraint}).
Following on this, we can sequentially update $(m-n+1)$ core tensors $\tX_{n},\tX_{n+1}, \ldots, \tX_{m}$, while fixing the other cores $\tX_j$, for $j<n$  or $j>m$. 
Since the cost function in (\ref{eq_cost_Gnm}) is formulated with orthogonality conditions on $\tX_j$, the new estimates $\tX_n, \tX_{n+1}, \ldots, \tX_m$ need to be orthogonalised accordingly in order to proceed to the next update.
Therefore, the algorithm should update the core tensors following the left-to-right order, i.e., increasing $n$, then switching to the right-to-left update procedure, i.e., decreasing $n$.

More specifically, in a single core update, for which $m = n$, the algorithm sequentially updates first the core tensors
$\tX_1, \tX_2, \ldots, \tX_{N-1}$, and then $\tX_N, \tX_{N-1}, \ldots, \tX_2$.
%

When $m = n+1$, the update can be with overlapping core indices, e.g.,
$(\tX_1, \tX_2)$, $(\tX_2, \tX_3)$, \ldots, as in the density matrix renormalization group (DMRG) optimization scheme \cite{White1993}.
This method sequentially optimises (reduces) ranks on the two sides of the core tensors, i.e., $R_1$, $R_2$, \ldots, $R_N$.
 When the tensor is of a relatively high order, say 20, the first core tensors tend to become small quickly in the first few iterations, while the ranks of the last core tensors remain relatively high. 
For such cases, updating ranks on only one side of the core tensors is recommended. 
For example, the update $(\tX_1, \tX_2)$, $(\tX_3, \tX_4)$ adjusts the ranks on the left side of $\tX_2$ and $\tX_4$.
Ranks on the right side of $\tX_2$ and $\tX_4$, i.e., $R_2$ and $R_4$, will be optimised when the algorithm runs the right-to-left update procedure, 
e.g., $(\tX_4, \tX_5)$, $(\tX_2, \tX_3)$.
Although the ranks $R_2$ and $R_4$ are not optimised in the left-to-right update, 
they are indeed not fixed, but adjusted due to the left-orthogonalization of $\tX_2$ and $\tX_4$.
Example~\ref{ex_TT_denoising_signal_2} compares the performance of the proposed algorithm over different numbers of overlapping core indices.


We also show that this update process is important in order to reduce computational costs in a progressive computation of the contracted tensors $\tT_{n:m}$, while for the particular cases of $m = n, n+1$ and $n+2$, we can derive efficient update rules for the core tensors $\tX_n, \tX_{n+1}, \ldots, \tX_m$.

\setlength{\textfloatsep}{.8em}
\setlength{\intextsep}{0pt} 
 \setlength{\algomargin}{.7em}
\begin{algorithm}[Ht!]
\SetFillComment
\SetSideCommentRight
\CommentSty{\footnotesize}
\caption{{\tt{The Alternating Multi-Cores Update Algorithm (AMCU)}}\label{alg_TT_progcomputing}}
\DontPrintSemicolon \SetFillComment \SetSideCommentRight
\KwIn{Data tensor $\tY$:  $(I_1 \times I_2 \times \cdots \times I_N)$,  and rank-$(R_1, R_2, \ldots, R_{N-1})$ or approximation accuracy $\varepsilon^2$,
\linebreak $k$ : the number of core tensors to be updated per iteration
$1 \le s \le k$ where $(k-s)$ indicates the number of overlapping core indices,
and $\widetilde{N}$ is the index of the first core to be updated in the right-to-left update
} 
\KwOut{TT-tensor $\tX = \tX_1 \bullet \tX_2 \bullet \cdots \bullet \tX_N$ of rank-$(R_1, R_2, \ldots, R_{N-1})$ such that $\min  \|\tY - \tX\|_F^2$ (or $\|\tY - \tX\|_F^2 \le \varepsilon^2$)}
\SetKwFunction{mreshape}{reshape}
\SetKwFunction{svds}{svds} 
\SetKwFunction{reshape}{reshape} 
\SetKwFunction{leftortho}{Left\_Orthogonalize} 
\SetKwFunction{rightortho}{Right\_Orthogonalize} 
\SetKwFunction{ttcpd}{TT\_CPD} 
\Begin{
\nl Initialize $\tX = \tX_1 \bullet \tX_2 \bullet \cdots \bullet \tX_N$, e.g, by rounding $\tY$\;
\Repeat{a stopping criterion is met}{
\mtcc{{\color{blue}{Left-to-Right update\dashrule}}}
\For{$n = 1, s+1, 2s+1, \ldots$}{
\mtcc{{\color{blue}{Tensor contraction in (\ref{equ_compute_Tnm})\dashrule}}}
\nl $\tT_{n:m} =  \tL_n \,\ltimesx_{N-m} \,  \tX_{>m}$\tcc*{$m = n+k-1$,$\tL_1 = \tY$}
\mtcc{{\color{blue}{Best TT-approximation to $\tT_{n:m}$\dashrule}}}
\nl $[\tX_n, \ldots, \tX_m]  =  \bestTT(\tT_{n:m})$\; 
\For{$i = n, n+1,\ldots, n+s-1$}{
\nl $\tX = \leftortho(\tX,i)$\;
\mtcc{{\color{blue}{Update left-side contracted tensor\dashrule}}}
\nl	$\tL_{i+1} =   \tX_i \, \ltimesx_2 \,  \tL_{i}$\label{step_update_Ln}\;
	}
}
\mtcc{{\color{blue}{Right-to-Left update\dashrule}}}
\For{$n = \widetilde{N}, \widetilde{N}-s,\widetilde{N}-2s, \ldots$}{
\nl $\tT_{n:m} =  \tL_n \,\ltimesx_{N-m} \,  \tX_{>m}$\;
\nl $[\tX_n, \ldots, \tX_m]  =  \bestTT(\tT_{n:m})$\; 
\For{$i = m, m-1,\ldots, m-s+1$}{
\nl $\tX = \rightortho(\tX,i)$
}
}
}
}
\end{algorithm} 
\subsection{A Progressive Computation of Contracted Tensors $\tT_{n:m}$}

The computation of the contracted tensors $\tT_{n:m}$ in (\ref{equ_compute_Tnm}), for $n = 1, 2, \ldots$, is the most computationally expensive step in Algorithm~\ref{alg_TT_progcomputing}, which requires 
$
\calO\left(\displaystyle\sum_{k = 1}^{n-1} R_{k-1}R_k\prod_{j = k}^{N} I_j \right)  
$ operations
for the left contraction $\tL_n  = \tX_{<n} \, \ltimesx_{n-1} \tY$, and 
$
\calO\left(\displaystyle R_n \sum_{k = m+1}^{N} R_{k-1}R_{k}\prod_{j = n}^{k} I_j \right)$ operations
for the right contraction $\tT_{n:m}  = \tL_{n } \, \rtimesx_{N-m}   \, \tX_{>m}$.  
For a particular case of $I_n = I$ and $R_n = R$ for all $n$, the computational cost to compute $\tT_n$ is of order $\calO(RI^{N} + R^2 I^{N-1})$. 

Since the left contraction $\tL_{n}$ can be expressed from $\tL_{n-1}$ as\\[-1.3em]
\begin{equation}
	\tL_{n} = \tX_{<n} \, \ltimesx_{n-1} \tY = \tX_{n-1} \ltimesx_{2} \tL_{n-1},		\notag 
\end{equation}\\[-1.6em]
where $\tL_1 = \tY$, the contracted tensors $\tT_{n:m}$ can be computed efficiently through a progressive computation of $\tL_n$.
Similarly, $\tT_{n:m}$ can also be computed through the right contracted tensors as
$\tT_{n:m}  =   \tX_{<n} \, \ltimesx_{n-1} \, \tR_m$, 
where $\tR_{m}  =  \tY  \, \rtimesx_{N-m}  \, \tX_{>m}  =   \tR_{m+1} \rtimesx_{2} \tX_{m+1}$.
In the left-to-right update procedure, the contracted tensors $\tT_{n:m}$ are computed from the left-side contracted tensors $\tL_n$. 
The tensors $\tL_{n+1}$, \ldots, $\tL_{n+s-1}$ for the next update are then computed sequentially from $\tL_n$ as in Step~\ref{step_update_Ln} in Algorithm~\ref{alg_TT_progcomputing}.
Here, $1\le s \le k$ while $(k-s)$ represents the number of overlapping core indices. 
When the algorithm is in the right-to-left update procedure, the left-side contracted tensors $\tL_n$ are available, and do not need to be computed.

A similar procedure can be implemented to exploit the right contracted tensors $\tR_{m}$ by first executing the right-to-left update procedure, then switching to the left-to-right update order.

This computation method is adapted from the alternating linear scheme \cite{Holtz-TT-2012,2013arXiv1301.6068D} or the two-site DMRG algorithm\cite{White1993,KressnerEIG2014} for solving linear systems or eigenvalue decompositions in which all variables are in the TT-format.
The alternating multi-cores update algorithm (AMCU) is briefly described in Algorithm~\ref{alg_TT_progcomputing}.
The routine $\bestTT$ within AMCU in Step~3 computes the best TT-approximation to $\tT_{n:m}$, which can be a low-rank matrix approximation or the low-multilinear rank Tucker-2 decomposition, depending on whether $m=n+1$ or $m=n+2$.
In general, the choice of it is free, but when $m=n$ (single core updates) the challenge becomes to find a rank-adaptive procedure for the denoising problem, as  discussed in the next section. The alternating double- and triple- cores update algorithms are presented in the Appendix.
\subsection{An Alternating Single Core Update (ASCU)}\label{sec:ascu_1}


We consider a simple case of the AMCU algorithm when $m = n$. The contracted tensor $\tT_n$ is then of size $R_{n-1} \times I_n \times R_{n}$,
and the error function in (\ref{eq_cost_Gnm}) becomes \\[-1.3em]
\be
D  =   \|\tY\|_F^2 - \|\tT_{n} \|_F^2  + \| \tT_{n} - \tX_{n} \|_F^2    \quad \text{for} \;\; n = 1, 2, \ldots, N.\label{eq_cost_Gn}
\ee
We can process the TT decomposition in two different ways

\begin{enumerate}[leftmargin=*,align=left,  itemindent=\dimexpr\labelsep+\labelwidth+7pt\relax]
\item {\bf A TT-approximation with a specified rank.} For this approximation problem, we obtain a solution $\tX_n = \tT_n$. 
\item {\bf A TT decomposition at a prescribed accuracy.} For the denoising problem, a new estimate of $\tX_n$ should have minimum ranks $R_{n-1}$ and $R_{n}$, such that 
\be
\| \tT_{n} - \tX_{n} \|_F^2    \le  \varepsilon_n^2\, \label{equ_Tn_Xn}
\ee
where $\varepsilon_n^2  = \varepsilon^2 - \|\tY\|_F^2 + \|\tT_{n} \|_F^2 $ is assumed to be non-negative. Note that adjusting the ranks $R_{n-1}$ and $R_n$ also requires manipulating  $\tX_{n-1}$ and $\tX_{n+1}$ accordingly, and $\tT_n$ implicitly depends on these manipulations. 
If a negative accuracy $\varepsilon_n^2$ occurs, this indicates that either the rank $R_{n-1}$ or $R_n$ is quite small, and needs to be increased, that is, the core $\tX_{n-1}$ or $\tX_{n+1}$ should be adjusted to have higher ranks. Often, the TT-rank $R_n$ is set to sufficiently high values, and then the TT-ranks $R_n$ will gradually decrease or at least behave in a non-increasing manner during the update of the core tensors. 
\end{enumerate}
%
%

It is not straightforward to update $\tX_n$ in the above problem; however, by expressing $\tX_n$ as a TT-tensor of three cores (\ref{eq_tucker2}), 
\be
	\tX_n = \bA_n \bullet \tilde{\tX}_n \bullet \bB_n		\notag 
\ee
the denoising problem in (\ref{equ_Tn_Xn}) reduces to finding a TT-tensor  $\bA_n \bullet \tilde{\tX}_{n} \bullet \bB_n$ which approximates $\tT_n$ with a minimum TT-rank-$(\tilde{R}_{n-1}, \tilde{R}_{n})$, such that 
\be
\| \tT_{n} - \bA_n \bullet \tilde{\tX}_{n} \bullet \bB_n \|_F^2    \le  \varepsilon_n^2  \, \notag \label{equ_Tn_Xn_2},
\ee
where $\bA_n$ and $\bB_n$ are matrices of size $R_{n-1} \times \tilde{R}_{n-1}$ and $\tilde{R}_{n} \times R_{n}$.

The TT-tensor $\bA_n \bullet \tX_{n} \bullet \bB_n $ can be estimated using the Tucker-2 decomposition in Section~\ref{sec:tt_order3}.
We note that the new estimate of $\tX$ is still of order-$N$ because the two cores $\bA_n$ and $\bB_n$ can be embedded into $\tX_{n-1}$ and $\tX_{n+1}$ as
\begin{align}
\tX 
&= \tX_1 \bullet \cdots \bullet (\tX_{n-1}  \bullet \bA_n)  \bullet \tX_n    \bullet (\bB_n \bullet \bX_{n+1}) \bullet \cdots \bullet \bX_N \notag \,.
\end{align}
In this way, the three cores $\tX_{n-1}$, $\tX_{n}$ and $\tX_{n+1}$ are  updated. Because $\bA_n$ and $\bB_n^T$ are orthogonal matrices, the newly adjusted cores $\tX_{n-1} \bullet \bA_n$ and $\bB_n \bullet \bX_{n+1}$ obey the left- and right orthogonality conditions. 
Algorithm~\ref{alg_TT_1coreupdate} outlines the single-core update algorithm based on the Tucker-2 decomposition.

 \setlength{\algomargin}{.7em}
\begin{algorithm}[Ht!]
\SetFillComment
\SetSideCommentRight
\CommentSty{\footnotesize}
\caption{{\tt{The Alternating Single-Core Update Algorithm {(\smaller two-sides rank adjustment)}}}\label{alg_TT_1coreupdate}}
\DontPrintSemicolon \SetFillComment \SetSideCommentRight
\KwIn{Data tensor $\tY$:  $(I_1 \times I_2 \times \cdots \times I_N)$ and  approximation accuracy $\varepsilon$
} 
\KwOut{TT-tensor $\tX = \tX_1 \bullet \tX_2 \bullet \cdots \bullet \tX_N$ of minimum TT-rank such that $\|\tY - \tX \|_F^2 \le \varepsilon^2$} \SetKwFunction{mreshape}{reshape}
\SetKwFunction{svds}{svds} 
\SetKwFunction{reshape}{reshape} 
\SetKwFunction{leftortho}{Left\_Orthogonalize} 
\SetKwFunction{rightortho}{Right\_Orthogonalize} 
\SetKwFunction{ttcpd}{TT\_CPD} 
\Begin{
\nl Initialize $\tX = \tX_1 \bullet \tX_2 \bullet \cdots \bullet \tX_N$\;
\Repeat{a stopping criterion is met}{
\mtcc{{\color{blue}{Left-to-Right update\dashrule}}}
\For{$n = 1,2 , \ldots, N-1$}{
\nl $\tT_{n} =    \tL_n   \rtimes_{N-n} \tX_{>n}$\;
\mtcc{{\color{blue}Solve Tucker-2 decomposition\dashrule}}
\nl $\|\tT_n - \bA_n  \bullet  \tX_n \bullet \bB_n\|_F^2  \le \varepsilon^2 - \|\tY\|_F^2 + \|\tT_n\|_F^2$\;
\mtcc{{\color{blue}Adjust adjacent cores\dashrule}}
\nl $\tX_{n-1} \leftarrow \tX_{n-1} \bullet \bA_n$,  $\tX_{n+1} \leftarrow  \bB_n\bullet \tX_{n+1}$\;
\nl $\tX = \leftortho(\tX,n)$\;
\mtcc{{\color{blue}Update left-side contracted tensors\dashrule}}
\nl $\tL_{n} \leftarrow \bA_n^T \bullet \tL_{n}$, $\tL_{n+1} \leftarrow \tX_n   \ltimes_2  \tL_{n}$
}
\mtcc{{\color{blue}{Right-to-Left update\dashrule}}}
\For{$n = N,N-1,\ldots, 2$}{
\nl $\tT_{n} =    \tL_n   \rtimes_{N-n} \tX_{>n}$\;
\nl $\|\tT_n - \bA_n  \bullet  \tX_n \bullet \bB_n\|_F^2  \le \varepsilon^2 - \|\tY\|_F^2 + \|\tT_n\|_F^2$\;
\nl $\tX_{n-1} \leftarrow \tX_{n-1} \bullet \bA_n$,  $\tX_{n+1} \leftarrow  \bB_n\bullet \tX_{n+1}$\;
\nl $\tX = \rightortho(\tX,n)$
}
}
}
\end{algorithm}


Alternatively, instead of adjusting the two ranks, $R_{n-1}$ and $R_{n}$, of $\tX_n$, we can update only one rank, either $R_{n-1}$ or $R_{n}$, corresponding to the right-to-left or left-to-right update order procedure. 
Assuming that the core tensors are updated in the left-to-right order, we need to find $\tX_n$ which has minimum rank-$R_{n}$ and satisfies  
\be
	\| \tT_{n}  - \tX_{n} \bullet \bB_n \|_F^2 \le  \varepsilon_n^2.  \notag \label{eq_adjust_right_rank}
\ee
This problem reduces to the truncated SVD of the mode-(1,2) matricization of $\tT_{n}$ with an accuracy  $\varepsilon_n^2$, that is 
\be
	[\tT_{n}]_{(1,2)} \approx  \bU_n \, \bSigma \, \bV_n^T  \,, \notag \label{eq_factorize_Tn}
\ee 
where $\bSigma = \diag(\sigma_{n,1}, \ldots, \sigma_{n,R_{n}^{\star}})$.
Here, for the new optimized rank $R_{n}^{\star}$, the following holds
\be
\sum_{r = 1}^{R_{n}^{\star}}  \, \sigma_{n,r}^2 \ge \|\tY\|_F^2 - \varepsilon^2 > \sum_{r = 1}^{R_{n}^{\star} -1}  \, \sigma_{n,r}^2\, . \label{eq_rank_Rn}
\ee 
The core tensor $\tX_n$ is then updated by reshaping $\bU_n$ to an order-3 tensor of size $R_{n-1} \times I_n \times R_{n}^{\star}$, while 
the core $\tX_{n+1}$ needs to be adjusted accordingly as 
\be
	\tX_{n+1}^{\star} =  \bSigma\, \bV_n^T \bullet 	\tX_{n+1}\, . \label{eq_adjust_xnplus1}
\ee
When the algorithm updates the core tensors in the right-to-left order, we update $\tX_n$ by using the $R_{n-1}^{\star}$ leading right singular vectors of the mode-1 matricization of $\tT_n$, and adjust the core $\tX_{n-1}$ accordingly, that is, \\[-1.6em]
\be
	[\tT_{n}]_{(1)} &\approx&  \bU_n \, \bSigma \, \bV_n^T  \, \notag\\[-.3em]
	\tX_n^{\star}  &=& \reshape(\bV_n^T, [R_{n-1}^{\star}, I_n , R_{n}])   \notag\\
	\tX_{n-1}^{\star}  &=& \tX_{n-1} \bullet \bU_n \bSigma\,. \label{eq_udpate_xnminus1}
\ee \\[-1.6em]
To summarise, the proposed method updates one core and adjusts (or rotates) another core. Hence, it updates two cores at a time.
 The new estimate $\tX_n^{\star}$ satisfies the left- or right-orthogonality condition, and does not need to be orthogonalised again. 
 The algorithm is listed in Algorithm~\ref{alg_TT_1coreupdate_b}.
 Another observation is that the tensor $\tX_{n+1}$ or  $\tX_{n-1}$ will be updated in the next iteration after updating $\tX_n$. Hence, the update of $\tX_{n+1}$ in (\ref{eq_adjust_xnplus1}), i.e., in Step 5, and  the update of $\tX_{n-1}$ in (\ref{eq_udpate_xnminus1}) , i.e., in Step 9, can be even skipped, except for the last update.

 \setlength{\algomargin}{.7em}
\begin{algorithm}[Ht!]
\SetFillComment
\SetSideCommentRight
\CommentSty{\footnotesize}
\caption{{\tt{The Alternating Single-Core Update Algorithm {(\smaller one side rank adjustment)}}}\label{alg_TT_1coreupdate_b}}
\DontPrintSemicolon \SetFillComment \SetSideCommentRight
\KwIn{Data tensor $\tY$:  $(I_1 \times I_2 \times \cdots \times I_N)$ and  accuracy $\varepsilon$
} 
\KwOut{TT-tensor $\tX = \tX_1 \bullet \tX_2 \bullet \cdots \bullet \tX_N$ of minimum TT-rank such that $\|\tY - \tX \|_F^2 \le \varepsilon^2$} \SetKwFunction{mreshape}{reshape}
\SetKwFunction{svds}{svds} 
\SetKwFunction{reshape}{reshape} 
\SetKwFunction{leftortho}{Left\_Orthogonalize} 
\SetKwFunction{rightortho}{Right\_Orthogonalize} 
\SetKwFunction{ttcpd}{TT\_CPD} 
\Begin{
\nl $\tX = \tX_1 \bullet \tX_2 \bullet \cdots \bullet \tX_N$ by rounding $\tY$\;
\Repeat{a stopping criterion is met}{
\mtcc{{\color{blue}{Left-to-Right update\dashrule}}}
\For{$n = 1,2 , \ldots, N-1$}{
\nl $\tT_{n} =    \tL_n   \rtimes_{N-n} \tX_{>n}$\;
  \nl  $[\tT_n]_{(1,2)} \approx \bU \, \bSigma \, \bV^T$\;
  \nl $\tX_n = \reshape(\bU, R_{n-1} \times I_n \times R_{n})$\;
  \mtcc{{\color{blue}{Adjust adjacent cores\dashrule}}}
\nl   $\tX_{n+1} \leftarrow  (\bSigma \, \bV^T) \bullet \tX_{n+1}$\;
  \mtcc{{\color{blue}{Update left-side contracted tensor\dashrule}}}
\nl $\tL_{n+1} \leftarrow \tX_n   \ltimes_2  \tL_{n}$
}
\mtcc{{\color{blue}{Right-to-Left update\dashrule}}}
\For{$n = N,N-1,\ldots, 2$}{
\nl $\tT_{n} =    \tL_n   \rtimes_{N-n} \tX_{>n}$\;
    \nl  $[\tT_n]_{(1)} \approx \bU \, \bSigma \, \bV^T$\;
   \nl $\tX_{n}  = \reshape(\bV^T, R_{n-1} \times I_n \times R_{n})$\;
    \nl $\tX_{n-1} \leftarrow \tX_{n-1} \bullet (\bU \, \bSigma)$\;
}
}
}
\end{algorithm}

\subsection{TT-SVD as a variant of ASCU with one update round}\label{sec:ttsvd_ascu}

Consider the approximation of a tensor $\tY$ of size $I_1 \times I_2 \times \cdots \times I_N$ using the ASCU algorithm with one-side rank adjustment at a given accuracy $\varepsilon^2$. Horizontal slices of the core tensors $\tX_n$ are initialized by unit vectors $\ve_r$ of length $I_n R_{n}$, as
$\vtr{\bX_n(r,:,:)} = \ve_r$, 
for $r = 1, 2, \ldots, R_{n-1}$, where the ranks $R_n$ are set to $R_n = \prod_{k = n+1}^{N}I_k$. 
For this initialization, the mode-1 matricizations of  the core tensors are identity matrices, $[\tX_n]_{(1)} = \bI_{R_{n-1}}$. Therefore, the contracted tensor $\tT_1$ is the data $\tY$, and the mode-1 approximation error is simply the global approximation error
$\varepsilon_1^2 = \varepsilon^2$.
For this reason, ASCU estimates the first core tensor $\tX_1$ as in TT-SVD. 

Since the core tensors $\tX_3$, \ldots, $\tX_N$ are not updated, the contracted tensor $\tT_2$ is the projection of $\tY$ onto the subspace spanned by $\tX_1$, implying that ASCU estimates $\tX_2$ in a similar way as TT-SVD. The difference here is that the mode-2 approximation accuracy $\varepsilon_2^2$ in ASCU is affected by the term $\|\tY\|_F^2 - \|\tT_2\|_F^2$.

The remaining core tensors $\tX_3$, \ldots, $\tX_N$ are updated similarly by ASCU and TT-SVD, but again the approximation accuracies in the two algorithms are different.
Another major difference is that TT-SVD estimates the core tensors once, i.e., by running only the left-to-right update (or the right-to-left update), while ASCU runs the right-to-left update after it completes the first round left-to-right update, and so on. This gradually either improves the approximation error or reduces the TT-rank of the estimated tensor. 

To summarise, the TT-SVD acts as an ASCU with one update round, but with a different error tolerance. As a result, ASCU yields a lower approximation error or smaller TT-ranks.
 
 \subsection{Comparison between the AMCU algorithms}\label{sec:compare_amcu}

Table~\ref{tab_compare_amcu} summarises the sub-optimisation problems of the ASCU, the alternating double-cores update (ADCU), and triple-cores update (ATCU) algorithms. In general, the ASCU with one-side rank adjustment (ASCU$_1$) works as the ADCU with one overlapping core index (ADCU$_1$), while the ASCU with two-sides rank adjustment (ASCU$_2$) updates the cores similarly to the updates of the ATCU with two overlapping core indices (ATCU$_2$). When the TT-rank is fixed, the ADCU with non-overlapping core indices (ADCU$_0$) is two times faster than the (ASCU$_1$), while 
ATCU$_0$ is faster than ADCU$_0$. However, the difference is significant only when the number of cores is large, i.e. tensors are of relatively high orders. More comparisons are  provided in Section~\ref{sec::examples}.
 
\begin{table}[t!]
\caption{Comparison of sub-optimisation problems per iteration between the AMCU algorithms. The ADCU$_k$ or ATCU$_k$ denotes the ADCU or ATCU algorithm with $k$ overlapping core indices, whereas ASCU$_k$ denotes the ASCU algorithm with $k$-sides rank adjustment.}
\label{tab_compare_amcu}
\begin{tabular*}{1\linewidth}{@{\extracolsep{\fill}}l@{\hspace{1ex}}l@{}l@{}}
AMCU &   Sub optimisation problems &  {Update order of core tensors}  \\\hline
{ASCU$_1$} &  \minitab[p{.5\linewidth}]{Low-rank matrix approximation to $\tT_n$}   &   \minitab[p{.45\linewidth}]{$\tX_n$, $\tX_{n+1}$   (work as ADCU$_{1}$)} \\\hline
{ASCU$_2$} &  \minitab[p{.5\linewidth}]{Low multilinear-rank Tucker-2 approximation to $\tT_n$}   &  \minitab[p{.45\linewidth}]{$\tX_{n-1}$ and  $\tX_n$ \hfill\\   (work as ATCU$_{2}$)} \\\hline
{ADCU }  &   \minitab[p{.5\linewidth}]{Low-rank matrix approximation to $\tT_{n,n+1}$}   &    \minitab[p{.45\linewidth}]{$\tX_n$ and $\tX_{n+1}$  
}   
\\\hline
{ATCU } &   \minitab[p{.5\linewidth}]{Low multilinear-rank Tucker-2 approximation to $\tT_{n,n+1,n+2}$}   &    \minitab[p{.45\linewidth}]{$\tX_{n}$, $\tX_{n+1}$ and  $\tX_{n+2}$
}  \\\hline
\end{tabular*}
\end{table}


\section{An Alternating Multi-Cores Update Algorithm for Input Tensor in TT-format}\label{sec:prior_compression}

Consider a data tensor $\tY$ given in the TT-tensor format,
which can be obtained by prior compression of data with higher accuracy tolerance using the TT-SVD algorithm. 
When tensors are given in the TT format, our alternating algorithms can be implemented with a much cheaper computational cost due to the efficient tensor contraction between two tensors $\tY$ and $\tX$. In other words, we assume that 
$\tY = \tY_1 \bullet \tY_2 \bullet \cdots \bullet \tY_N$, where  $\tY_n$ are of size $S_{n-1} \times I_n \times S_{n}$. We shall next  introduce fast contractions between two TT-tensors, followed by a formulation of  update rules for the AMCU algorithm.

\subsection{The contraction between TT-tensors}

As previously stated, the most computationally expensive step in the AMCU algorithms is to compute the contraction tensors $\tT_{n:m}$.
For two TT-tensors $\tY$ and $\tX$,  we then have \\[-1.5em]
\be
\tT_{n:m} &=&  ( \tX_{<n}  \ltimesx_{n-1}  \tY)  \rtimesx_{N-m} \tX_{>m} \, \notag \\
&=&   (\tX_{<n} \ltimesx_{n-1} \tY_{<n})  \bullet  \tY_{n:m}  \bullet  (\tY_{>m}  \rtimesx_{N-m} \tX_{>m}) \, \notag \\
&=& \bPhi_{n} \bullet  \tY_{n:m}  \bullet  \bPsi_{m} \notag \label{equ_Tn_TT}
\ee\\[-1.5em]
where the matrices $\bPhi_{n}$ are of size $R_{n-1} \times S_{n-1}$, and represent a left-contraction between $\tX_{<n}$ and $\tY_{<n}$ along the first ($n-1$) modes, and the matrices $\bPsi_{n}$ are of size $S_{n} \times R_{n}$, and represent a right contraction between $\tY_{>n}$ and $\tX_{>n}$ along all but mode-1, \\[-1.3em]
\be
\bPhi_{n} &=  \tX_{<n} \, \ltimesx_{n-1} \,  \tY_{<n}     , \quad 
\bPsi_{n} =  \tY_{>n} \, \rtimesx_{N-n}  \, \tX_{>n}   \, . \notag 
\ee \\[-1.3em]
The contraction matrices $\bPhi_{n}$ and $\bPsi_{n}$ can be efficiently computed  as \\[-1.5em]
\be
\bPhi_{n+1}  &=& ( \tX_{<n} \bullet \tX_{n}) \ltimesx_{n-1} (\tY_{<n} \bullet \tY_{n})   \notag \\
&=&   \tX_{n} \ltimesx_2  (\bPhi_{n} \bullet \tY_{n} ) 	\, ,	\notag 
\\
\bPsi_{n-1}  &=& ( \tY_{n} \bullet \tY_{>n}) \rtimesx_{N-n} (\tX_{n} \bullet \tX_{>n})  \notag \\
&=&   (\tY_{n} \bullet  \bPsi_{n})  \rtimesx_2   \tX_{n}      \, .\notag 
\ee
with the respective complexities of $\calO(I_n R_{n-1} S_{n} (R_{n} + S_{n-1}))$
and $\calO(I_n R_{n} S_{n-1} (S_{n} + R_{n-1}))$,


\subsection{A Generalized Framework for the AMCU algorithm}\label{sec::ASCU_tt_tensor}

Similar to the alternating multi-core update in Algorithm~\ref{alg_TT_progcomputing}, the algorithm for the TT-tensor $\tY$ is summarised in Algorithm~\ref{alg_TT_dcoreupdate_2}.
It is important to emphasise that we do not update the right and left contraction matrices $\bPhi_n$ and $\bPsi_n$ when updating the core tensors, 
but update either $\bPhi_{n+1}$ or $\bPsi_{n-1}$.
In order to achieve this, we first compute the right contraction matrices  $\bPsi_{n}$ before entering the main loop. Here, we denote $\bPsi_N = \bPhi_1= 1$.
At the first iteration, the algorithm executes the left-to-right update procedure, and estimates $\tX_{1:k}$ 
as the best TT approximation to the tensor $\tY_{1:k} \bullet \bPsi_{k} $. The core tensors are then orthogonalized, and 
the left-contraction matrices $\bPhi_{2}, \bPhi_{3},\ldots, \bPhi_{s}$ are updated accordingly. Similarly, the algorithm computes the new core tensors $\tX_n, \tX_{n+1}, \ldots, \tX_{n+s-1}$, left-
orthogonalizes them, then updates the left-contraction matrices $\bPhi_{n}$ without computing the right contraction matrices $\bPsi_{n}$. 

While running the right-to-left update, the algorithm does not need to compute the left-contraction matrices but it updates the right-contraction matrices $\bPsi_{n-1}$, \ldots,  $\bPsi_{n+k-2}$.

\setlength{\intextsep}{.1ex}
\begin{figure*}[t!]
\centering
%
\psfrag{Relative Error}[cc][cc]{\scalebox{1}{\color[rgb]{0,0,0}\setlength{\tabcolsep}{0pt}\begin{tabular}{c}\smaller\smaller\smaller \vspace{1em}Relative Error\end{tabular}}}
\psfrag{A2CU}[lb][lb]{\scalebox{.9}{\color[rgb]{0,0,0}\setlength{\tabcolsep}{0pt}\begin{tabular}{c}\smaller\smaller\smaller\smaller\smaller\tiny ADCU\end{tabular}}}
\subfigure[Relative Errors as a function of iterations]{\includegraphics[width=.32\linewidth,trim = 0.0cm 0cm -.20cm .0cm,clip=false]{fig_ex1_error_vs_iters}
\label{fig_ex1_error_vs_iters}}
%
%
\psfrag{MSAE}[cc][cc]{\scalebox{1}{\color[rgb]{0,0,0}\setlength{\tabcolsep}{0pt}\begin{tabular}{c}\smaller\smaller\smaller \vspace{1ex}MSAE (dB)\end{tabular}}}%
\subfigure[Mean squared angular errors]{\includegraphics[width=.30\linewidth, trim = 0cm 0.0cm 0cm 0cm,clip=false]{fig_mc_ex1_sae}
\label{fig_mc_ex1_sae}}
\hfill
%
%
\psfrag{SVD1}[rt][rt]{\scalebox{.8}{\color[rgb]{0,0,0}\setlength{\tabcolsep}{0pt}\begin{tabular}{r} \footnotesize TT-SVD\\[-.5em] \footnotesize+Rounding\end{tabular}}}
\psfrag{SVD2}[lt][lt]{\scalebox{.8}{\color[rgb]{0,0,0}\setlength{\tabcolsep}{0pt}\begin{tabular}{l}\\[-.7em]\smaller\smaller\smaller \footnotesize TT-SVD\\[-.5em] \footnotesize given rank\end{tabular}}}
%
%
\subfigure[Execution time]{\includegraphics[width=.34\linewidth, trim = 1cm 0.0cm 1.6cm 0cm,clip=true]{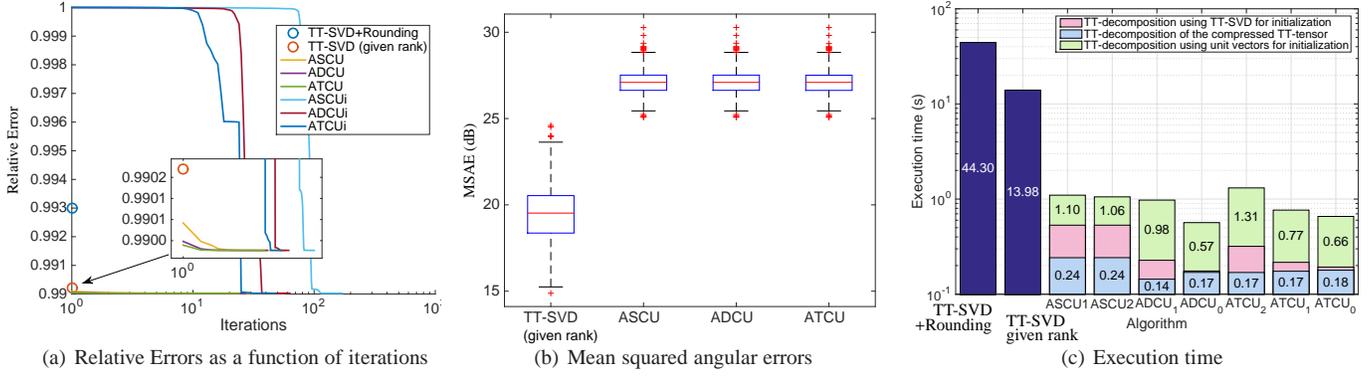}\label{fig_ex1c_len16777216_SNR-20_exectime}}
\vspace{-1ex}
\caption{Denoising of signals at poor SNR.
\subref{fig_ex1_error_vs_iters} Illustration of the convergence behaviour of the AMCU algorithms when the signal length $K = 2^{22}$. A closer inspection shows that the AMCU algorithms achieve lower approximation errors than TT-SVD. \subref{fig_mc_ex1_sae} mean squared angular errors over 500 independent runs when $K = 2^{22}$;  \subref{fig_ex1c_len16777216_SNR-20_exectime} Execution time in seconds when the signal length $K = 2^{24}$.}
\label{fig_ex_TT_reconstruct_exp}
\vspace{-1ex}
\end{figure*}

 \setlength{\algomargin}{.7em}
\begin{algorithm}[Ht!]
\SetFillComment
\SetSideCommentRight
\CommentSty{\footnotesize}
\caption{{\tt{The Alternating Multi-Cores Update Algorithm for TT-tensor}}\label{alg_TT_dcoreupdate_2}}
\DontPrintSemicolon \SetFillComment \SetSideCommentRight
\KwIn{TT-tensor $\tY = \tY_1 \bullet \tY_2 \bullet \cdots \bullet \tY_N$, and  approximation accuracy $\varepsilon$, $\widetilde{N}$ is the index of the first core to be updated in the right-to-left update, 
 $k$ is the number of core tensors to be updated per iteration, and 
$1 \le s \le k$ where $(k-s)$ indicates the number of overlapping indices,

} 
\KwOut{TT-tensor $\tX = \tX_1 \bullet \tX_2 \bullet \cdots \bullet \tX_N$ such that $\|\tY - \tX \|_F^2  \le \varepsilon$ with lower TT-ranks} \SetKwFunction{mreshape}{reshape}
\SetKwFunction{svds}{svds} 
\SetKwFunction{reshape}{reshape} 
\SetKwFunction{leftortho}{Left\_Orthogonalize} 
\SetKwFunction{rightortho}{Right\_Orthogonalize} 
\SetKwFunction{bestTT}{bestTT\_approx} 
\Begin{
\nl Initialize $\tX = \tX_1 \bullet \tX_2 \bullet \cdots \bullet \tX_N$ by rounding $\tY$\;
\mtcc{{\color{blue}{Precompute the right-contracted matrices $\bPsi_n$}\dashrule}}
\For{$n = N-1, \ldots, 1$}{
\nl $    \bPsi_{n} = (  \tY_{n+1} \bullet  \bPsi_{n+1})  \rtimesx_{2}   \tX_{n+1}      $\tcc*{$\bPsi_{N} = 1$}}
\Repeat{a stopping criterion is met}{
\mtcc{{\color{blue}{Left-to-Right update}\dashrule}}
\For{$n = 1, s+1, 2s+1, \ldots$}{
\mtcc{{\color{blue}{Contracted tensor $\tT_{n:m}$}, $m = n+k-1$\dashrule}}
\nl $\tT_{n:m} = \bPhi_{n}  \bullet  \tY_n \bullet \tY_{n+1} \bullet \cdots \bullet \tY_m \bullet \bPsi_{m}$\;
\mtcc{{\color{blue}{Best TT-approximation to $\tT_{n:m}$}\dashrule}}
\nl $[\tX_n, \ldots, \tX_m] = \bestTT(\tT_{n:m},\varepsilon)$\;
\For{$i = n, n+1,\ldots, n+s-1$}{
\nl $\tX = \leftortho(\tX,i)$ 
\nl $\bPhi_{i+1} \leftarrow \tX_{i}  \ltimesx_2  (\bPhi_{i} \bullet \tY_{i})$   
}
}
\mtcc{{\color{blue}{Right-to-Left update}\dashrule}}
\For{$n = \widetilde{N}, \widetilde{N}-s,\widetilde{N}-2s, \ldots$}{
\nl $\tT_{n:m} = \bPhi_{n}  \bullet  \tY_n \bullet \tY_{n+1} \bullet \cdots \bullet \tY_m \bullet \bPsi_{m}$\;
\nl $[\tX_n, \ldots, \tX_m] = \bestTT(\tT_{n:m},\varepsilon)$\;
\For{$i = m, m-1,\ldots, m-s+1$}{
\nl $\tX = \rightortho(\tX,i)$
\nl $ \bPsi_{i-1} \leftarrow (  \tY_{i} \bullet  \bPsi_{i})   \rtimesx_2 \tX_{i}$  
}
}
}
}
\end{algorithm} 

For $k = 1$, the single core update algorithm updates $\tX_n$ as in Section~\ref{sec:ascu_1}. For $k = 2$, the alternating double cores update (ADCU) algorithm computes a low rank approximation to the mode-(1,2) unfolding of the contracted tensor\\[-1.5em]
\be
\tT_{n,n+1} = \bPhi_{n} \bullet  \tY_{n} \bullet \tY_{n+1} \bullet  \bPsi_{n+1} \, ,  \notag 
\ee\\[-1.5em] or in other words, a truncated SVD of the following matrix\\[-1.5em]
\be
[\tT_{n,n+1}]_{(1,2)} =   [\bPhi_{n} \bullet  \tY_{n}]_{(1,2)} \,    [\tY_{n+1} \bullet  \bPsi_{n+1} ]_{(1)}  \,   \notag 
\ee\\[-1.5em]
where $[\bPhi_{n} \bullet  \tY_{n}]_{(1,2)}$ and $[\tY_{n+1} \bullet  \bPsi_{n+1} ]_{(1)} $ are respective of sizes $R_{n-1} I_n \times S_{n}$ and $S_{n} \times I_{n+1} R_{n+1}$. 
When $S_{n} <  R_{n-1}I_n$ and $S_{n} < I_{n+1} R_{n+1}$, the SVD is computed for a reduced size matrix $\bU_n \bV_n^T$, where $\bU_n$ and $\bV_n$ are the upper triangular matrices in the QR decompositions of $[\bPhi_{n} \bullet  \tY_{n}]_{(1,2)}$ and $ [\tY_{n+1} \bullet  \bPsi_{n+1} ]_{(1)}^T$.

For the alternating triple cores  update algorithm, the tensor contractions are computed for three indices $[n, n+1, n+2]$ as \\[-1.5em]
\be
\tT_{\bn} &=& \bPhi_{n} \bullet  \tY_{n} \bullet \tY_{n+1} \bullet \tY_{n+2} \bullet  \bPsi_{n+2} .		\notag 
\ee \\[-1.5em]
The algorithm solves the Tucker-2 decomposition of the mode-(1,2),3,(4,5) unfolding of $\tT_{\bn}$ as (see (\ref{eq_tucker2}))
\be
\min_{\bU_n, \bV_n}  \|\tZ_n - \bU_n \bullet \tX_{n+1} \bullet \bV_n^T \|_F^2		\notag 
\ee
where $\bU_n = [\tX_n]_{(1,2)}$,  $\bV_n = [\tX_{n+2}]_{(1)}^T$, and\\[-1.5em]
\be
\tZ_n = [\tT_{\bn}]_{(1,2),3,(4,5)} = \bA_n \bullet \tY_{n+1}   \bullet   \bB_n^T		\notag 
\ee\\[-1.5em]
and $\bA_n = [\bPhi_{n} \bullet  \tY_{n} ]_{(1,2)}$ are of size $R_{n-1}I_n \times S_{n}$, 
while 
$\bB_n = [\tY_{n+2} \bullet  \bPsi_{n+2}]_{(2,3)}$ are of size $I_{n+2} R_{n+2} \times S_{n+1}$.
The two factor matrices, $\bU_n$ and $\bV_n$, are sequentially estimated as principal components of the matrices 
$(\tZ \bullet \bV_n) \rtimesx_2 (\tZ \bullet \bV_n)$ and the matrices $(\bU_n^T \bullet \tZ)   \ltimesx_2 (\bU_n^T \bullet \tZ)$.


\section{Simulations}\label{sec::examples}

%
%
%
%
%
%
%

We first validated the proposed algorithms through two examples on the denoising of exponentially decaying signals which admit the TT representation.
Second, our method was tested on the denoising of benchmark color images.
 For this application, a novel tensorization was developed to construct
order-5 tensors from small patches of the images. The final example considers blind source separation from a single channel mixture.

\begin{example}{\bf The reconstruction of known target ranks.}\label{ex_TT_reconstruction_exp} Harmonic retrieval is at the very core of signal processing applications. To illustrate the potential of the TT decomposition in this context, we considered the reconstruction of an exponentially decaying signal $x(t)$ from a noisy observation $y(t)$ of $K = 2^{d}$ samples, where $d$ = 22, 24 or 26, given by\\[-1.5em]
\be
y(t) &=&   x(t)+ e(t)		\notag 
\ee\\[-1.7em]
and \\[-1.7em]
\be
x(t) = \exp({\frac{-5t}{K}}) \, \sin(\frac{2\pi f}{f_s} t  + \frac{\pi}{3}) 	\label{eq_x5}
\ee
with $f  = 10$ Hz, $f_s = 100$ Hz, while $e(t)$ represents the additive Gaussian noise, which was randomly generated such that the signal noise ratio SNR = -20 dB.

The observed signal was tensorized (reshaped) to an order-($d-2$) tensor $\tY$ of size $4 \times 2 \times \cdots \times 2 \times 4$. With this tensorization, the sinusoid  yields a TT-tensor of rank-(2,2,\ldots,2), whereas the signal $exp(t)$ yields a rank-1 tensor. Hence, its Hadamard product, i.e., $x(t)$, admits a TT-representation of rank-$(2,2,\ldots, 2)$\cite{OseledetsTT11}, and gives the TT-model\\[-1.5em]
\be
\tY = \tX + \tE \, ,	\notag 
\ee\\[-1.5em]
where $\tX$ is the TT-tensor of the signal $x(t)$, and $\tE$ is reshaped from the noise.
In other words, we attempted to approximate the tensor $\tY$ by a TT-tensor with a prior known TT-rank.

In order to compare the widely-used TT-SVD algorithm with our proposed AMCU algorithm, the tensor $\tY$ was first decomposed using the TT-SVD such that \\[-1.5em]
\be
\|\tY - \hat{\tX}\|_F^2 \le  \varepsilon^2	\notag 
\ee\\[-1.5em]
where $\varepsilon$  is a measure of the added noise. For this decomposition, the TT-SVD yielded TT-tensors with quite high ranks, which exceeded the TT-rank of $\tX$. The results were then ``rounded'' to the TT-rank of $\tX$ \cite{OseledetsTT09}. For the two-stage decomposition, we used the TT-tensor toolbox \cite{oseledets2014tt}.

Alternatively, to obtain a TT-tensor having the same ranks as $\tX$, the TT-SVD algorithm computed only $R_n =2$ leading singular vectors from the projected data. The outcome TT-tensor was used to initialise the AMCU algorithms.

We ran the simulation 500 times,
and assessed performance through the relative error and the squared angular error given respectively by\\[-1.7em]
\be
\delta(\by, \hat{\bx}) &=& \frac{\|\by - \hat{\bx}\|_2^2}{\|\by\|_2^2} \,, \notag \\[-1ex]
SAE(\bx,\hat{\bx}) &=& -20\log_{10}\arccos \frac{\bx^T {\hat{\bx}}}{\|\bx\|_2 \|\hat{\bx}\|_2}  \quad (dB).\notag 
\ee\\[-1.3em]
Fig.~\ref{fig_ex1_error_vs_iters} compares the convergence behaviour of the AMCU algorithms over one run when $K = 2^{22}$. 
The TT-SVD with rounding achieved an approximation error of 0.9930, while 
given TT-ranks, it yielded a TT-tensor with a lower approximation error of 0.9902. With this result as the initial value, the AMCU algorithms improved the  approximation error to 0.9900. A similar result was achieved when the AMCU algorithms were initialised by a TT-tensor, the $n$-th core of which is given in the form $\vtr{\hat{\tX}_n(r,:,:)} = \ve_r$.
The AMCU algorithms converged after a dozen iterations for the first initialisation, and required more iterations for the latter initialisation method, denoted by AMCU$_i$. 

%
Fig.~\ref{fig_mc_ex1_sae} illustrates a performance comparison in terms of SAEs for the case $K= 2^{22}$, showing on average that the signals reconstructed by our proposed algorithms exhibit an 8 dB higher SAE than when using the TT-SVD with the rank specified. 
For $K = 2^{24}$ and $2^{26}$, the average SAEs of the TT-SVD were improved to 25.70 and 29.07 dB, but were still lower than the respective mean SAEs of 32.56 and 38.17  dB achieved using our algorithms. 

%

\begin{figure}[t]
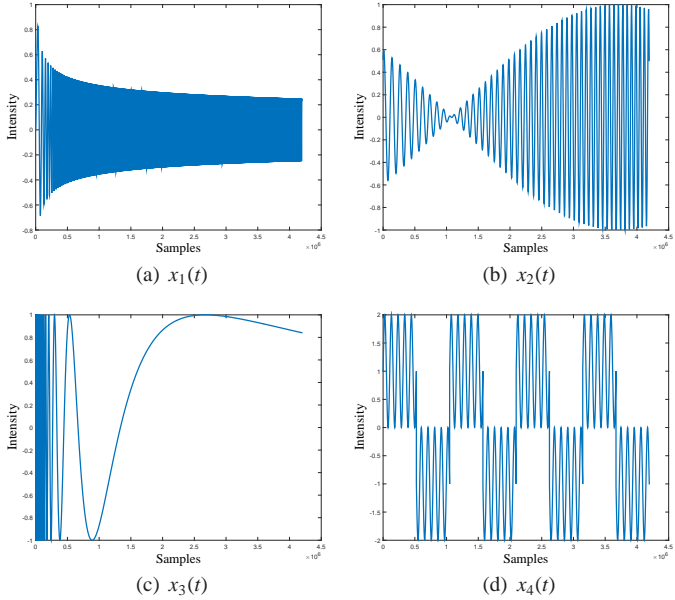

\centering
\psfrag{Intensity}[cc][cc]{\scalebox{1}{\color[rgb]{0,0,0}\setlength{\tabcolsep}{0pt}\begin{tabular}{c}\smaller\smaller\smaller\smaller \vspace{1em}Intensity\end{tabular}}}%
\psfrag{Sammples}[cc][cc]{\scalebox{1}{\color[rgb]{0,0,0}\setlength{\tabcolsep}{0pt}\begin{tabular}{c}\smaller\smaller\smaller\smaller  Samples\end{tabular}}}%
\subfigure[$x_1(t)$]{\includegraphics[width=.47\linewidth, trim = 0cm 0cm 0cm 0cm,clip=true]{fig_tt_denoising_x_ex11.eps}
\label{fig_tt_denoising_x_ex11}}
\hfill
\subfigure[$x_2(t)$]{\includegraphics[width=.47\linewidth, trim = 0cm 0cm 0cm 0cm,clip=true]{fig_tt_denoising_x_ex3}
\label{fig_tt_denoising_x_ex3}}
\hfill
\subfigure[$x_3(t)$]{\includegraphics[width=.47\linewidth, trim = 0cm 0cm 0cm 0cm,clip=true]{fig_tt_denoising_x_ex2}
\label{fig_tt_denoising_x_ex2}}
\hfill
\subfigure[$x_4(t)$]{\includegraphics[width=.47\linewidth, trim = 0cm 0cm 0cm 0cm,clip=true]{fig_tt_denoising_x_ex4}
\label{fig_tt_denoising_x_ex4}}
\caption{Original signals $x(t)$ used in Example~\ref{ex_TT_denoising_signal_2}.}
\label{fig_ex_TT_denoising_signal}
\end{figure}

For completeness, Fig.~\ref{fig_ex1c_len16777216_SNR-20_exectime} compares the execution times of the considered algorithms, where ASCU$_1$ and ASCU$_2$ denote the ASCU algorithms with one and two sides rank adjustment, respectively, while ADCU$_k$ and ATCU$_k$ indicate the ADCU and ATCU algorithms with $k$ overlapping core indices, where $k = 0, 1,  2$. 
When the signal length $K = 2^{24}$, the TT-SVD with rounding took an average execution time of 44.30 seconds on a computer based on Intel Xeon E5-1650, clocked at 3.50 GHz and with 64 GB of main memory.
For a given TT-rank, this algorithm worked faster, and completed the approximation in 13.98 seconds. Since the outcomes of TT-SVD were good initial values, the ASCU algorithm needed only 0.53 seconds, while the ADCU and ATCU algorithms were approximately two times faster than the ASCU. Even when the core tensors were initialised by unit vectors $\ve_r$, the proposed algorithms converged very quickly (in less than 2 seconds), i.e., much shorter  than the total execution times when these algorithm were initialised by TT-SVD. For this kind of initialisation, the ATCU was on average the fastest, and ASCU the slowest algorithm.


Finally, we illustrate performance of the AMCU algorithms in Algorithm~\ref{alg_TT_dcoreupdate_2} for the task of fitting the TT-tensors $\tY_{\tilde{\varepsilon}}$, which were approximations to $\tY$ with an accuracy of $\tilde{\varepsilon} = 0.3$, using the TT-SVD. The algorithms achieved an average SAE of 26.95 dB when the signal length $K = 2^{22}$ and an SAE of 32.52 dB when $K = 2^{24}$. There was no significant loss in accuracy compared to the AMCU fit to the tensor $\tY$. Moreover, the AMCU algorithm required shorter running times, e.g,  0.24 seconds for the ASCU algorithm, and 0.17 seconds for the ADCU and  ATCU  algorithms.

\end{example}

\begin{table*}[ht!]
\caption{The TT-ranks of signals $x_r$ of length $K = 2^{22}$ and of their estimates $\hat{x}_r$ using the TT-SVD and the AMCU algorithms in Example~\ref{ex_TT_denoising_signal_2}. The squared angular error is given on the logarithmic scale, and the execution time is in seconds.
}\label{tab_ex_2_ranks_}
\centering
\vspace{-1ex}
\begin{tabular}{llrcc}
Signal &  \multicolumn{1}{c}{TT-ranks} & SAE (dB)  & Time (s)\\
\hline
$x_1$ 			& 2-2-3-3-3-3-4-4-5-6-7-8-10-13-19-26-32-16-8-4-2    &  & \\
$\hat{x}_{TT-SVD}$  &2-4-8-16-31-59-112-210-387-677-967-789-443-228-115-58-30-16-8-4-2   	&   4.18  & 9.69\\
$\hat{x}_{ASCU}$     & 1-1-1-1-1-1-1-1-2-2-3-3-6-11-20-34-32-16-8-4-2    & 27.49 & 3.25\\
$\hat{x}_{ADCU_1}$ & 1-1-1-1-1-1-1-2-2-3-5-8-14-28-49-45-24-16-8-4-2   &    26.66  &  2.01 \\
$\hat{x}_{ADCU_0}$ & 1-2-1-2-1-2-1-2-2-4-5-10-13-26-20-40-24-16-8-4-2    & 27.89 & 2.54\\				   
$\hat{x}_{ATCU_2}$  & 1-1-1-1-1-1-1-2-2-3-5-8-14-28-48-22-24-16-8-4-2  &    27.61  & 2.81\\
$\hat{x}_{ATCU_1}$  & 1-1-1-1-1-1-1-2-2-3-5-8-14-26-37-22-24-16-8-4-2  &    27.64  & 2.41\\
$\hat{x}_{ATCU_0}$  & 1-1-1-1-1-1-1-2-4-3-5-10-14-22-33-22-24-16-8-4-2   & 28.18 & 2.62\\
\hline
$x_2$ 		        & 2-4-8-16-32-56-47-38-32-26-22-18-15-13-12-10-8-7-6-4-2   & \\
$\hat{x}_{TT-SVD}$  & 2-4-8-16-31-59-112-210-387-675-959-782-440-226-114-57-28-15-8-4-2   &   6.18  & 9.63\\
$\hat{x}_{ASCU}$     & 1-1-1-1-1-1-1-2-4-8-13-21-35-65-92-54-27-15-8-4-2      & 22.73 & 2.08\\
$\hat{x}_{ADCU_1}$ & 1-1-1-1-1-1-2-3-5-8-12-20-37-71-94-52-26-13-7-4-2   & 23.10    & 1.52\\
$\hat{x}_{ADCU_0}$ & 1-2-1-2-1-2-2-4-4-8-11-22-36-72-85-54-27-15-8-4-2       & 23.34 & 1.45\\
$\hat{x}_{ATCU_2}$ & 1-1-1-1-1-1-2-3-5-8-12-20-37-70-104-56-28-13-7-4-2   & 23.11  & 1.83 \\
$\hat{x}_{ATCU_0}$ & 1-1-1-1-1-1-1-2-4-7-11-22-37-66-105-54-27-15-8-4-2      & 23.07  & 1.67 \\
\hline
$x_3$ & 2-2-2-2-2-3-3-3-3-4-4-4-5-6-7-9-12-16-8-4-2   &   \\
$\hat{x}_{TT-SVD}$  & 2-4-8-16-31-59-112-210-387-677-966-789-443-228-115-58-29-15-8-4-2   & 4.41 & 9.67\\
$\hat{x}_{ASCU}$     & 1-1-1-1-1-1-1-1-1-1-2-2-2-3-4-7-11-13-8-4-2   & 31.48	& 3.10\\ 
$\hat{x}_{ADCU_1}$ & 1-1-1-1-1-1-1-1-1-2-2-3-6-10-18-32-16-8-8-4-2   &32.47	& 2.15\\
$\hat{x}_{ADCU_0}$ & 1-2-1-2-1-2-1-2-1-2-2-4-2-4-6-12-16-16-8-4-2   &34.58	& 2.52\\
$\hat{x}_{ATCU_2}$ & 1-1-1-1-1-1-1-1-1-2-2-3-5-8-14-23-8-8-8-4-2   &33.52	&  2.77\\
$\hat{x}_{ATCU_0}$ & 1-1-2-1-1-2-1-1-2-1-2-3-2-3-6-6-9-16-8-4-2   &31.49	& 2.58\\
\hline
$x_4$ 			& 2-2-2-2-2-3-3-3-3-3-3-3-3-3-3-3-3-2-1-1-1    & \\
$\hat{x}_{TT-SVD}$  & 2-4-8-16-31-59-111-207-378-653-920-762-433-223-112-56-28-14-7-4-2 & 4.36  & 9.69 \\
$\hat{x}_{ASCU}$     & 1-1-1-1-1-1-1-1-1-1-1-2-2-3-3-3-3-2-1-1-1   & 35.88 	& 2.91\\
$\hat{x}_{ADCU_1}$ & 1-1-1-1-1-1-1-1-1-1-2-2-2-3-4-8-15-11-7-4-2   & 35.89	& 2.17\\
$\hat{x}_{ADCU_0}$ & 1-2-1-2-1-2-1-2-1-2-1-2-4-8-13-26-24-16-8-4-2   & 39.35 & 2.50\\
$\hat{x}_{ATCU_0}$ & 1-1-2-1-1-2-1-1-2-1-1-2-2-3-3-3-3-2-1-1-1   & 36.12 	& 2.61\\
\hline
$x_5$ 			& 2-2-2-2-2-2-2-2-2-2-2-2-2-2-2-2-2-2-2-2-2    & \\
$\hat{x}_{TT-SVD}$  &  2-4-8-16-31-59-112-210-387-677-966-788-443-228-115-58-29-15-8-4-2 & 5.17  & 9.69 \\
$\hat{x}_{ASCU}$      & 2-2-2-2-2-2-2-2-2-2-2-2-2-2-2-2-2-2-2-2-2   & 46.04	& 3.21\\
$\hat{x}_{ADCU_1}$  & 2-2-2-2-2-2-2-2-2-2-2-2-2-2-2-2-2-2-2-2-2   & 46.04	& 3.17\\
$\hat{x}_{ATCU_0}$  & 2-2-2-2-2-2-2-2-2-2-2-2-2-2-2-2-2-2-2-2-2   & 46.04 	& 2.76\\
\hline
\end{tabular}
\end{table*}

\begin{example}{\bf Denoising with unknown target ranks.}\label{ex_TT_denoising_signal_2}
To illustrate utility of the TT decomposition as a tool for denoising, we considered noisy signals $y(t) = x(t)  +  e(t)$, degraded versions of a signal $x(t)$ through contamination with additive Gaussian noise $e(t)$, 
where $x(t)$ can take one of the following forms (see Fig.~\ref{fig_ex_TT_denoising_signal})\\[-1.5em]
\begin{align}
x_1(t) &= \frac{\sin(2000 t^{2/3})}{4 t^{1/4}} \, ,  
& x_3(t) &= \sin(\frac{5(t+1)}{2})  \cos(100 (t+1)^2)  ,
 \notag \\
x_2(t) &=  \sin(t^{-1})  \,,  
&x_4(t) &=  \sign(\sin(8\pi t)) (1 + \sin(80\pi t))  ,\notag 
\end{align}
or the damped signal used in Example~\ref{ex_TT_reconstruction_exp}.
The signals $y(t)$ in our example had length $K = 2^{22}$, and were tensorized (reshaped) to tensors $\tY$ of order-22 and size $2 \times 2 \times \cdots \times 2$. 
With this tensorization, the five signals $x_r(t)$ can be well approximated by tensors in the TT-format, with their TT-ranks given in Table~\ref{tab_ex_2_ranks_}, where $x_5(t)$ is the signal in (\ref{eq_x5}).


\begin{figure}[t]
\centering
\psfrag{T}[lc][lc]{\scalebox{1}{\color[rgb]{0,0,0}\setlength{\tabcolsep}{0pt}\begin{tabular}{l} \vspace{1ex}\smaller\smaller TT-SVD\end{tabular}}}%
\psfrag{Alternating}[lc][cl]{\scalebox{1}{\color[rgb]{0,0,0}\setlength{\tabcolsep}{0pt}\begin{tabular}{l}\vspace{1ex}  \smaller\smaller AMCU\end{tabular}}}%
\psfrag{Noisy}[lc][lc]{\scalebox{1}{\color[rgb]{0,0,0}\setlength{\tabcolsep}{0pt}\begin{tabular}{l}\vspace{1ex}  \smaller\smaller $y(t)$\end{tabular}}}%
\psfrag{Source}[lc][lc]{\scalebox{1}{\color[rgb]{0,0,0}\setlength{\tabcolsep}{0pt}\begin{tabular}{l}\vspace{1ex}  \smaller\smaller $x(t)$\end{tabular}}}%
{\includegraphics[width=.45\textwidth, trim = 0cm 0cm 0cm 1.5cm,clip=true]{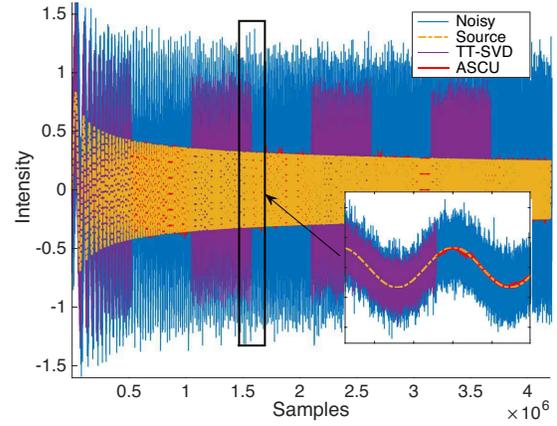}
\label{fig_tt_denoising_sr10B_ex1}}
\vspace{-1ex}
\caption{Illustration of the noisy signal $x_1(t)$ in Example~\ref{ex_TT_denoising_signal_2}, and the signals estimated using the TT-SVD and ASCU.}
\label{fig_ex_TT_denoising_est_signal}
\end{figure}

We applied the alternating single and multi-cores update algorithms to approximate the noisy tensor, with the results plotted in Fig.~\ref{fig_ex_TT_denoising_est_signal} for the test case with SNR = 0 dB. 
Fig.~\ref{fig_ex_TT_denoising_est_signal} plots 8 non overlapping segments of the 
the signals estimated using the TT-SVD and our alternating algorithm.
The approximation $\|\tY - \hat{\tX}\|_F^2 \le \varepsilon^2$ was first performed using the TT-SVD algorithm, where the accuracy level of $\varepsilon^2 = \sigma^2 K$, and $\sigma$ is the standard deviation of the Gaussian noise. The reconstructed signals achieved SAEs of 4.18, 6.18, 4.43, 4.36 and 5.17 dB for the five signals $x_r(t)$, respectively. 
When using the ASCU, ADCU and ATCU algorithms, much better performances were obtained with average respective SAEs =   33.11,  33.49 and  33.23 dB. 
The performance comparison is presented in Table~\ref{tab_ex_2_ranks_}, where ADCU$_1$ and ADCU$_0$ denote the performances of the ADCU algorithms with one overlapping index and non overlapping indices, respectively.
For example, for the reconstruction of the signal $x_1{(t)}$, ADCU$_1$ enforced the first eight core tensors to be quite small with a rank of 1, and could not suppress the TT-ranks of the last core tensors. 
Consequently, the TT-ranks $R_{15}$, $R_{16}$, $R_{17}$ and $R_{18}$ exceeded those of ADCU$_{0}$,
and the TT-tensor estimated by ADCU$_1$ had 11578 entries, which was more than the 6798 entries estimated by ADCU$_{0}$.
Another important observation is that the signal reconstructed by ADCU$_1$ was worse than the reconstruction by ADCU$_{0}$, by about 1dB SAE.

Besides higher angular errors, the TT-SVD yielded approximations with TT-ranks significantly higher than those of the sources. This detrimental effect did not happen for the ASCU algorithm.
For this example, the TT-SVD algorithm took on average 9.67 seconds to estimate all the core tensors of the five tensors $\hat{\tX}_r$, while the ADCU and ATCU algorithms needed  2.24 and   2.37 seconds, respectively, and were slightly faster than ASCU.

\end{example}

\begin{example}{\bf Image denoising.}\label{ex_TT_denoising_image}
We next tested the proposed algorithms in a novel application of the TT-decomposition for image denoising.  Given that the intensities of pixels in a small window are highly correlated, our method was able to learn hidden structures which represent relations between small patches of pixels. 
These structures are then used to reconstruct the image as a whole.

For a color image $\bY$ of size $I  \times J \times 3$ degraded by additive Gaussian noise, the basic idea behind the proposed method is that for each block of pixels of size $h \times w  \times 3$, given by\\[-1.5em]
\be
	\bY_{r,c} = \bY(r:r+h-1,c:c+w-1,:)		\notag 
\ee\\[-1.5em]
a small tensor $\tY_{r,c}$ of size $h \times w \times 3 \times (2d+1) \times (2d+1)$, comprising $(2d+1)^2$ blocks around $\bY_{r,c}$ is constructed, in the form 
\be
	\tY_{r,c}(:,:,:,d+1+i,d+1+j) =   \bY_{r+i,c+j}, 		\notag 
\ee
where $i, j = -d, \ldots, 0 , \ldots, d$,  and $d$ represents the neighbour width.
Every $(r,c)$-block $\bY_{r,c}$ is then approximated through the TT-decomposition 
\be
	\|\tY_{r,c} - \tX_{r,c}\|_F^2 \le \varepsilon^2 \label{equ_block_approx}
\ee
where the noise level $\varepsilon^2$ can be determined by inspecting the coefficients of the image in the high frequency bands.
A pixel is then reconstructed as the average of all its approximations by TT-tensors which cover that pixel. 

\begin{figure}
\includegraphics[width=.158\linewidth, trim = 0.0cm 0cm 0cm 0cm,clip=true]{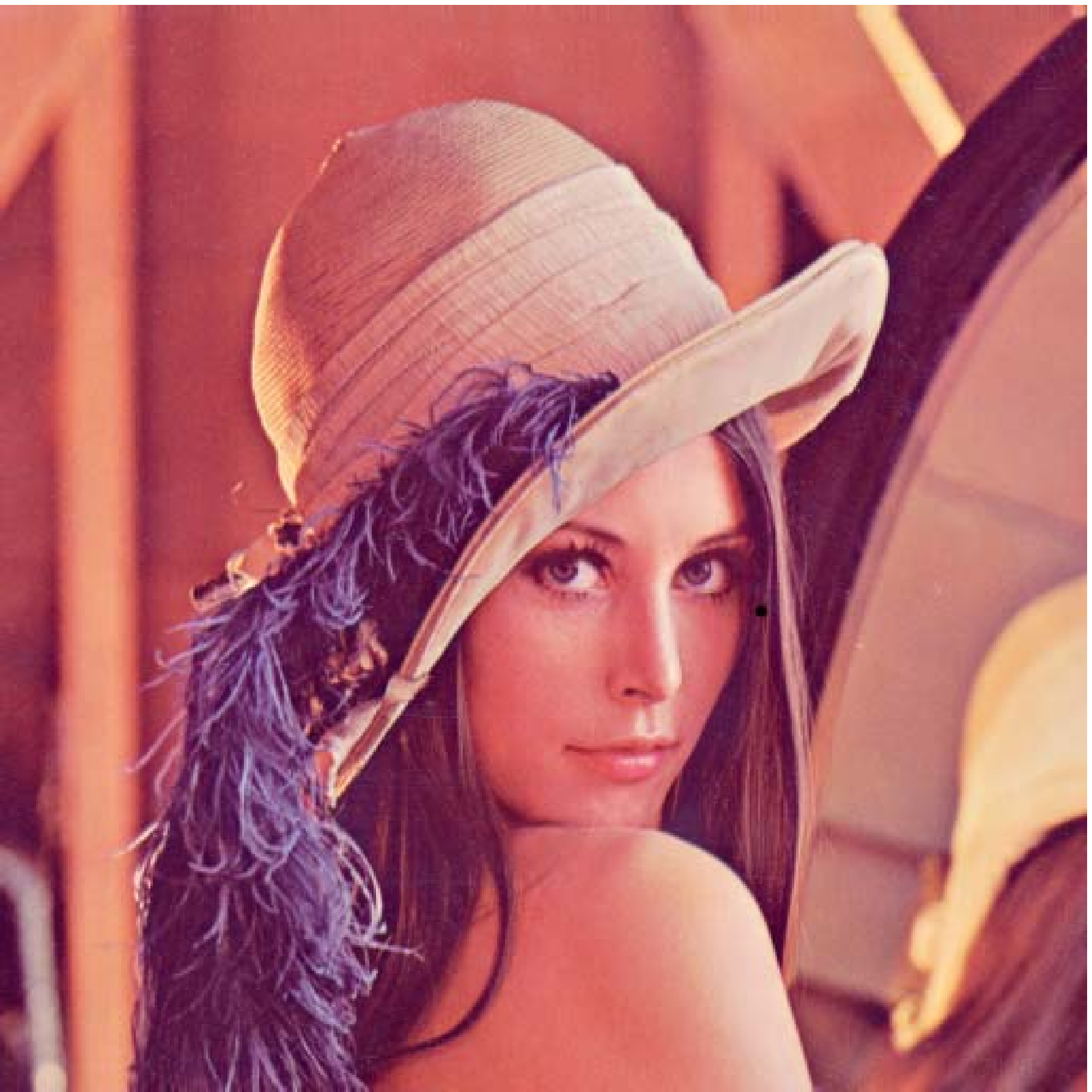}
\includegraphics[width=.158\linewidth, trim = 0.0cm 0cm 0cm 0cm,clip=true]{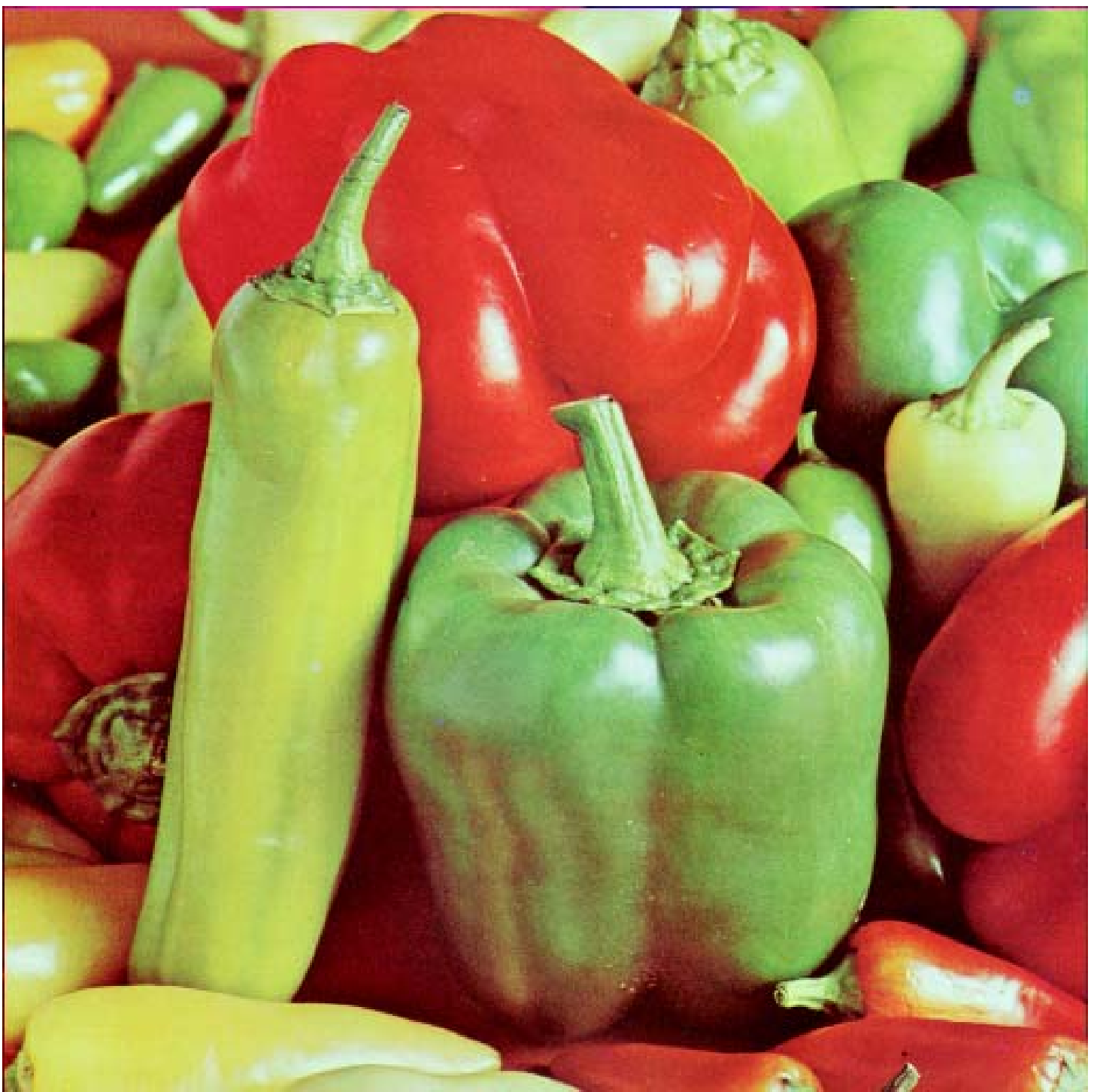}
\includegraphics[width=.158\linewidth,height=.158\linewidth, trim = 0.0cm 0cm 0cm 0cm,clip=true]{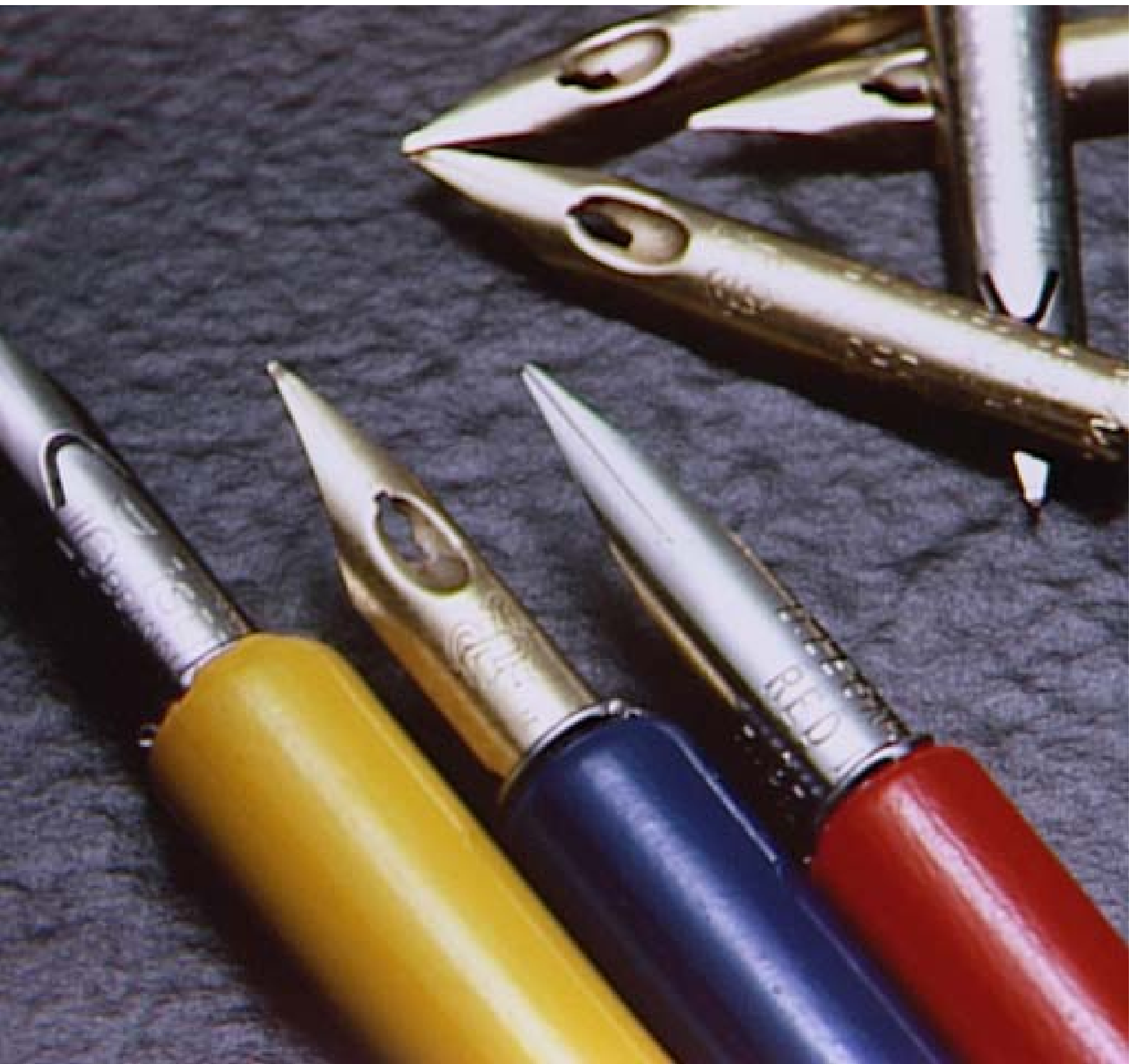}
\includegraphics[width=.158\linewidth, trim = 0.0cm 0cm 0cm 0cm,clip=true]{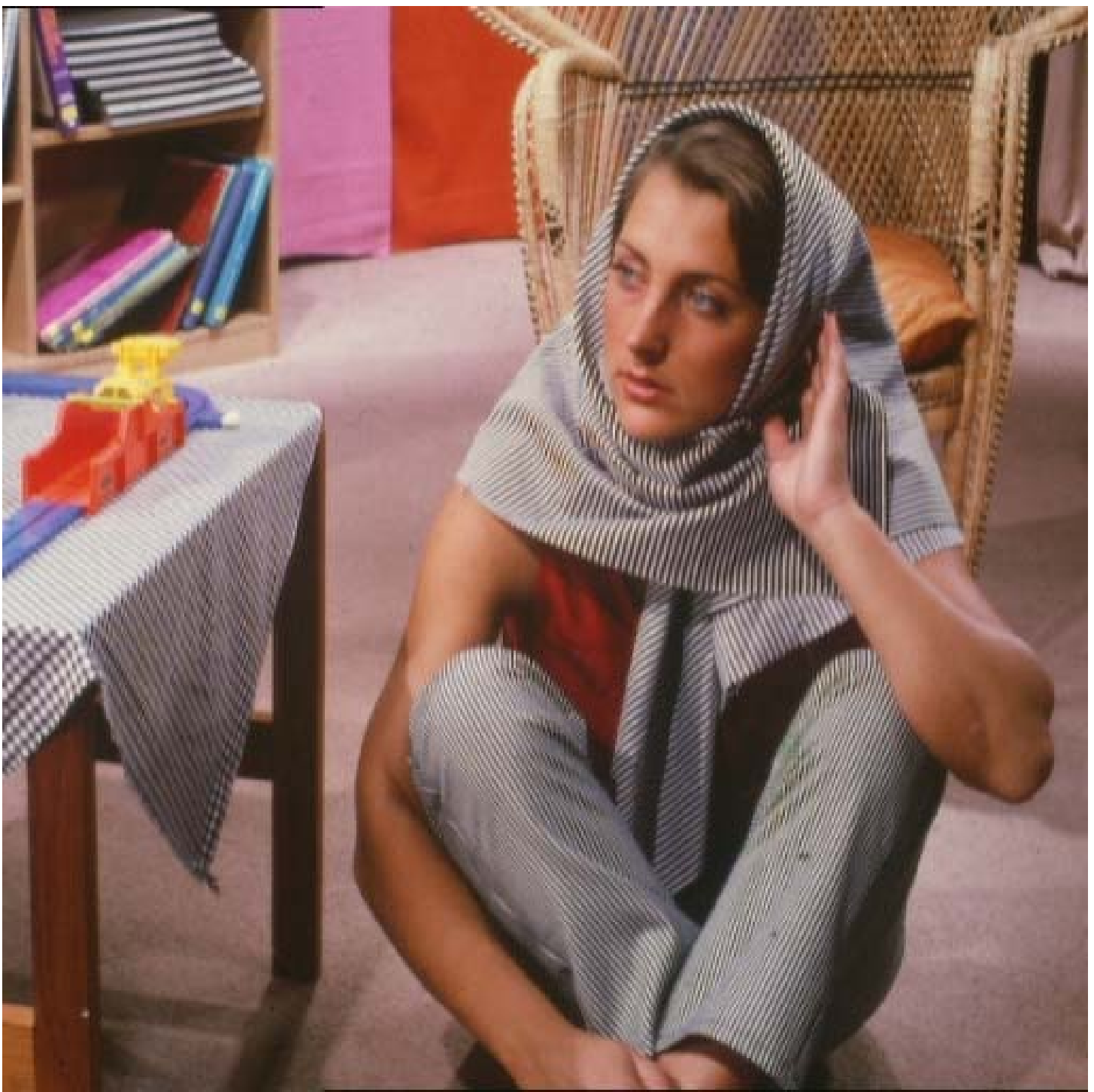}
\includegraphics[width=.158\linewidth, trim = 0.0cm 0cm 0cm 0cm,clip=true]{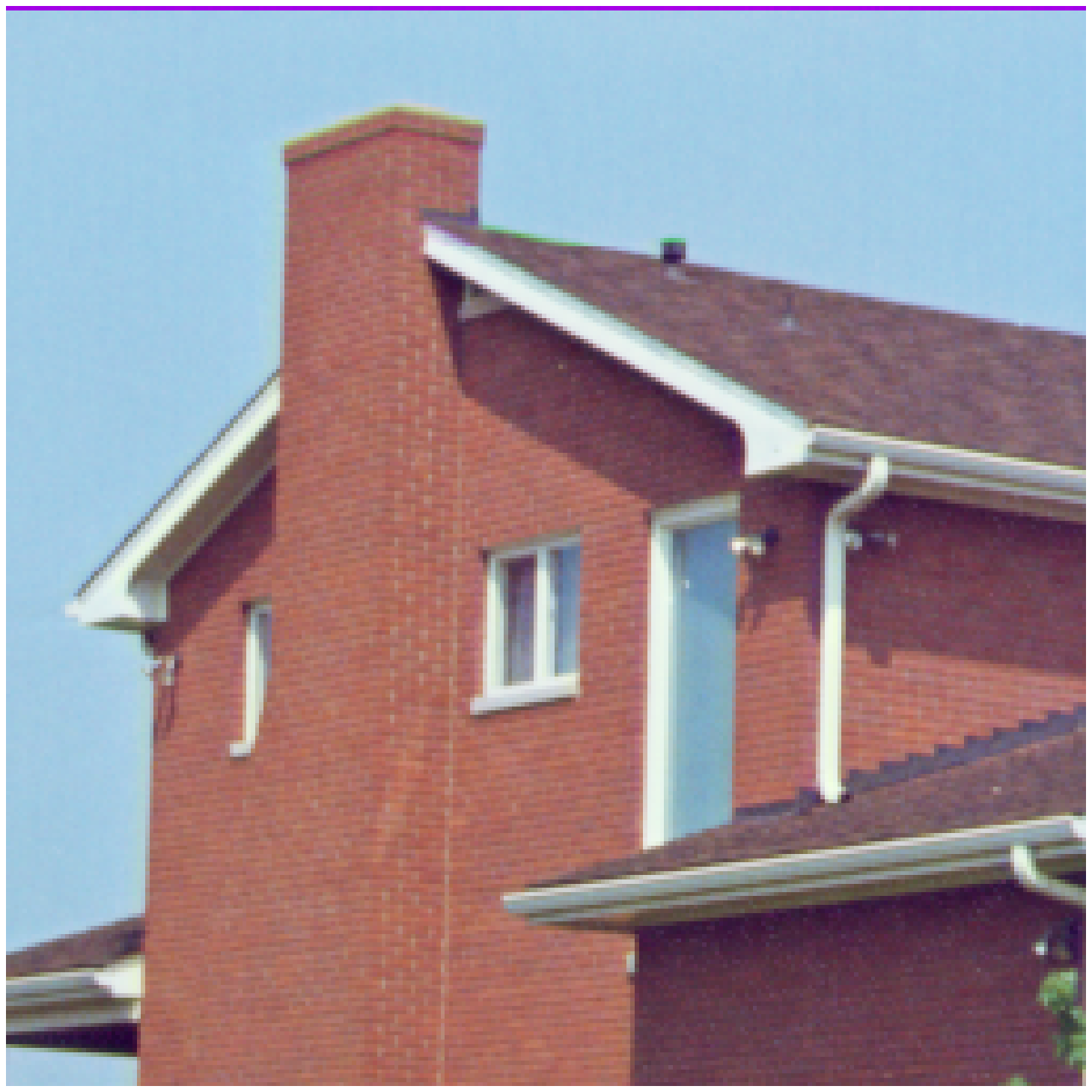}
\includegraphics[width=.158\linewidth, trim = 0.0cm 0cm 0cm 0cm,clip=true]{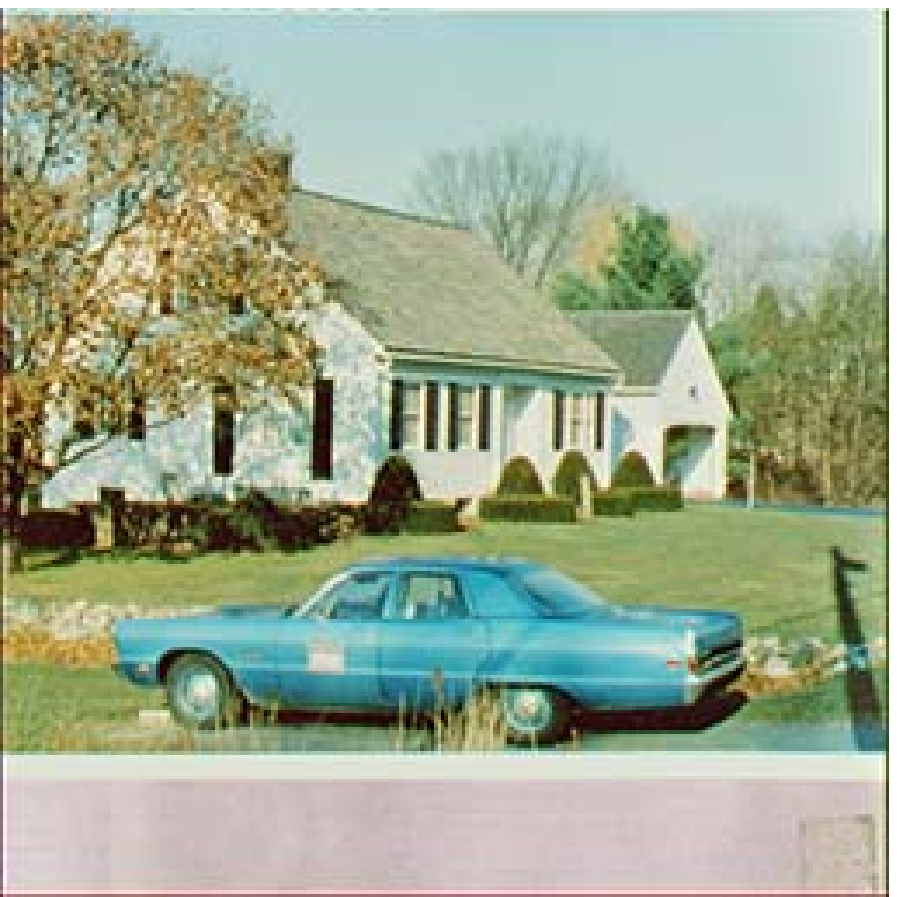}
\caption{Six images of size $256 \times 256$ are used in Example~\ref{ex_TT_denoising_image}.}
\label{fig_6images}
\end{figure}

In our simulations, we used six benchmark color images of size $256 \times 256 \times 3$ (illustrated in Fig.~\ref{fig_6images}), and corrupted them white Gaussian noise at SNR = 10 dB.
Latent structures were learnt for patches of sizes $8 \times 8 \times 3$ (i.e., $h = w = 8$) in the search area of width $d = 3$. 
To the noisy images, we applied the DCT spatial filtering before the block reconstruction.
For the block approximation problem in (\ref{equ_block_approx}), we applied several tensor decompositions, including
the TT-SVD, the Tucker approximation (TKA) with a predefined approximation error, the Bayesian Robust tensor factorisation (BRTF) for low-rank CP decomposition \cite{DBLP:journals/pami/ZhaoZC15},  and the alternating single core update algorithm (ASCU).
The algorithm for TKA works in a similar way to the Tucker-2 algorithm in Section~\ref{sec:tt_order3}, but for an order-5 tensor, and estimates 5 factor matrices.
In addition, we recovered the image with sparsity constraints using a dictionary of 256 atoms learnt by K-SVD\cite{Aharon2006}.
For this method, three layers of color images were flattened to an array of size $256 \times 768$. The dictionary was learnt for patches of size $8 \times 8$.

The quality of images reconstructed by five different methods was assessed using three indices:  mean-squared error (MSE), peak signal-to-noise ratio (PSNR), and the structural similarity index (SSIM). The results are shown in Table~\ref{tab_image_denosing}, and illustrated in Figs.~\ref{fig_lena_10dB}-\ref{fig_pepper_10dB}.
By learning similarities between patches, our proposed method was able to  recover the image, and achieved better performance than the well-known denoising method based on dictionary learning. Moreover, the results confirm the superiority of our proposed ASCU algorithm over the TT-SVD algorithm, and over other tensor decompositions. 
Using the ADCU algorithm, we obtained comparable performances to those of ASCU.

Finally, Fig.~\ref{fig_tt_rank_maps} visualizes the TT-rank maps of tensors approximated using ASCU,
overlaid by edges of the images.  
Each entry of the map represents the average of the sum of the TT-ranks of tensors which reconstruct the pixel at the same location.
{\emph{It is clear that the TT-ranks of the blocks containing the details were high, while they are low at flat regions that contain no details}}.

\begin{table}[t]
\caption{The performance comparison of algorithms considered in Example~\ref{ex_TT_denoising_image}  in terms of MSE (dB), PSNR (dB) and SSIM for image denoising when SNR = 10 dB.}\label{tab_image_denosing}
\centering
\begin{tabular}{lc@{}cccc@{}ccccccccc}
Algorithms    	&&      MSE	&       PSNR      	& SSIM    && MSE &       PSNR   	& SSIM \\
\hline
 &&  \multicolumn{3}{c}{Lena}&&  \multicolumn{3}{c}{Pepper}    \\\cline{3-5}\cline{7-9} 
TT-SVD 	&& 	  35.11   &    32.68   &    0.892  				&&   40.40   &    32.07   &    0.861    \\
TT-ASCU		&& 	  \bf{27.37}   &    \bf{33.76}   &    \bf{0.927}   &&    \bf{31.47}   &    \bf{33.15}   &    {0.924}   \\
TT-ADCU  	&&	  28.04   &   33.65    &    0.926  			&&   32.09   &    33.07 & 0.923\\
Tucker    		&&	  34.59   &    32.74   &    0.919 			&&   38.96   &    32.23   &    0.917   \\
BRTF    		&&	  40.30   &    32.07   &    0.840  			&&   46.85   &    31.42   &    0.825   \\
K-SVD  		&&	  34.76   &    32.72   &    0.908  			&&  35.74   &    32.60   &    0.918   \\\hline
%
%
 &&  \multicolumn{3}{c}{Pens}&&  \multicolumn{3}{c}{Barbara}     \\\cline{3-5}\cline{7-9} 
TT-SVD 		&&   44.92   & 	31.61   & 	 0.884    &&   32.30   &    33.04   &    0.901  \\
TT-ASCU		&&   \bf{36.61}   &	\bf{32.50}   &	 \bf{0.908}   &&   \bf{24.92}   &    \bf{34.16}   &    \bf{0.934}   \\
Tucker    		&&   48.56   &	31.27   &	 0.884    &&   33.20   &    32.92   &    0.919   \\
BRTF    		&&   42.80   &	31.82   &	 0.877  &&   31.87   &    33.10   &    0.899   \\
K-SVD  		&&   50.04   & 	31.14   &	 0.862  &&   35.41   &    32.64   &    0.908   \\ \hline
%
 && \multicolumn{3}{c}{House}   &&  \multicolumn{3}{c}{House2}     \\\cline{3-5}\cline{7-9} 
TT-SVD   &&   23.70    &    34.38   & 0.877				&& 	41.07 & 32.00	& 0.905		\\
TT-ASCU &&   {19.30}   &    {35.28}   &    \bf{0.899}   	&&\bf{38.53} & \bf{32.27}	& \bf{0.926}	\\
Tucker     &&    23.64   &    34.40   &    0.885 			&& 	48.11 & 31.31	& 0.909		\\
BRTF      &&    27.93   &    33.67   &    0.823   			&&	42.99 & 31.80   & 0.867		\\
K-SVD    &&    22.18   &    34.67   &    0.881  			&& 	46.44 & 31.46	& 0.907 \\	\hline
    \end{tabular}
    \end{table}

\begin{figure}[t!]
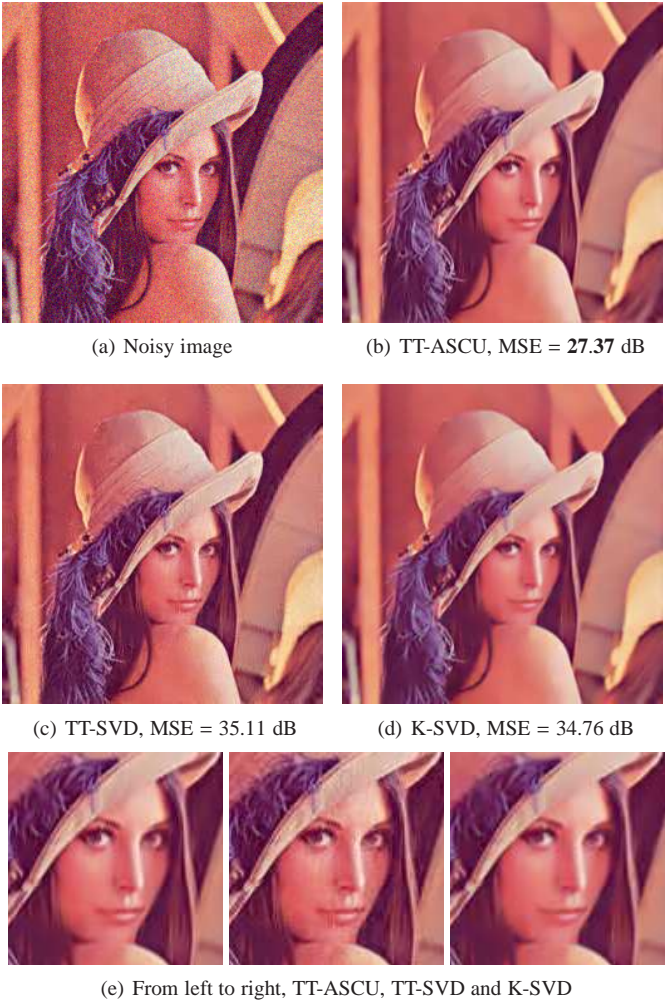

\centering
%
%
%
\subfigure[Noisy image]{\includegraphics[width=.48\linewidth, trim = 0cm 0cm 0cm 0cm,clip=true]{fig_lena256_10dB}}
\hfill 
\subfigure[TT-ASCU,  MSE = $\bf 27.37$ dB]{\includegraphics[width=.48\linewidth, trim = 0cm 0cm 0cm 0cm,clip=true]{fig_lena256_10dB_tt_ascu}
\label{fig_lena256_10dB_tt_ascu}}
%
\hfill
\subfigure[TT-SVD, MSE = 35.11 dB]{\includegraphics[width=.48\linewidth, trim = 0cm 0cm 0cm 0cm,clip=true]{fig_lena256_10dB_tt_truncation}
\label{fig_lena256_10dB_tt_truncation}}
%
%
\hfill
\subfigure[K-SVD, MSE = 34.76 dB]{\includegraphics[width=.48\linewidth, trim = 0cm 0cm 0cm 0cm,clip=true]{fig_lena256_10dB_ksvd}
\label{fig_lena256_10dB_ksvd}}
\subfigure[From left to right, TT-ASCU, TT-SVD and K-SVD]{
\includegraphics[width=.32\linewidth, trim = 4cm 2.5cm 2.5cm 4cm,clip=true]{fig_lena256_10dB_tt_ascu}
\includegraphics[width=.32\linewidth, trim = 4cm 2.5cm 2.5cm 4cm,clip=true]{fig_lena256_10dB_tt_truncation}
\includegraphics[width=.32\linewidth, trim = 4cm 2.5cm 2.5cm 4cm,clip=true]{fig_lena256_10dB_ksvd}
}
\caption{The Lena image corrupted by noise at 10 dB SNR, and the patches reconstructed by different methods in Example~\ref{ex_TT_denoising_image}.}
\label{fig_lena_10dB}
\end{figure}

\begin{figure}[ht!]
\centering
\subfigure[Lena]{\includegraphics[width=.47\linewidth, trim = 0cm 0cm 0cm 0cm,clip=true]{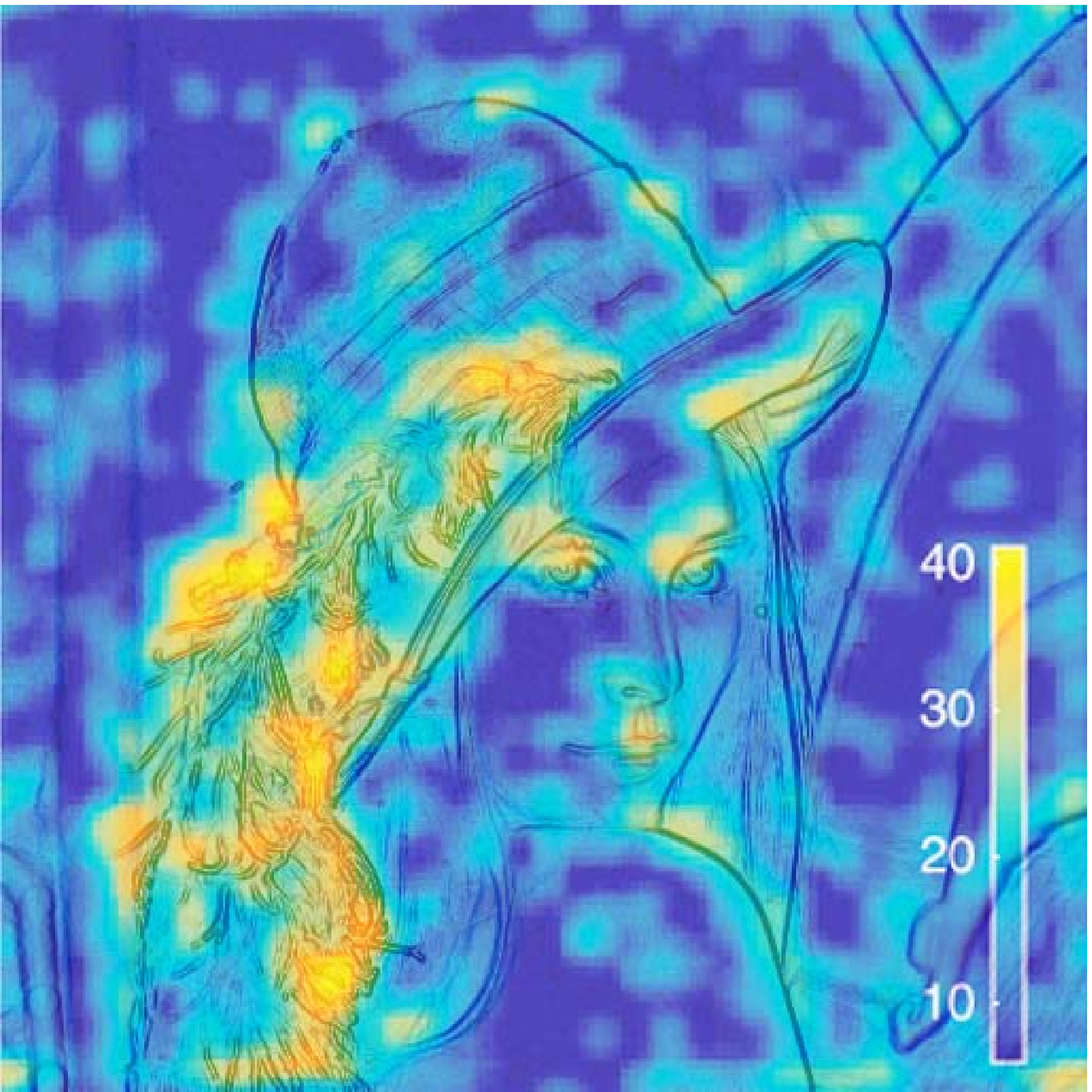}}
\hfill 
\subfigure[Pepper]{\includegraphics[width=.47\linewidth, trim = 0cm 0cm 0cm 0cm,clip=true]{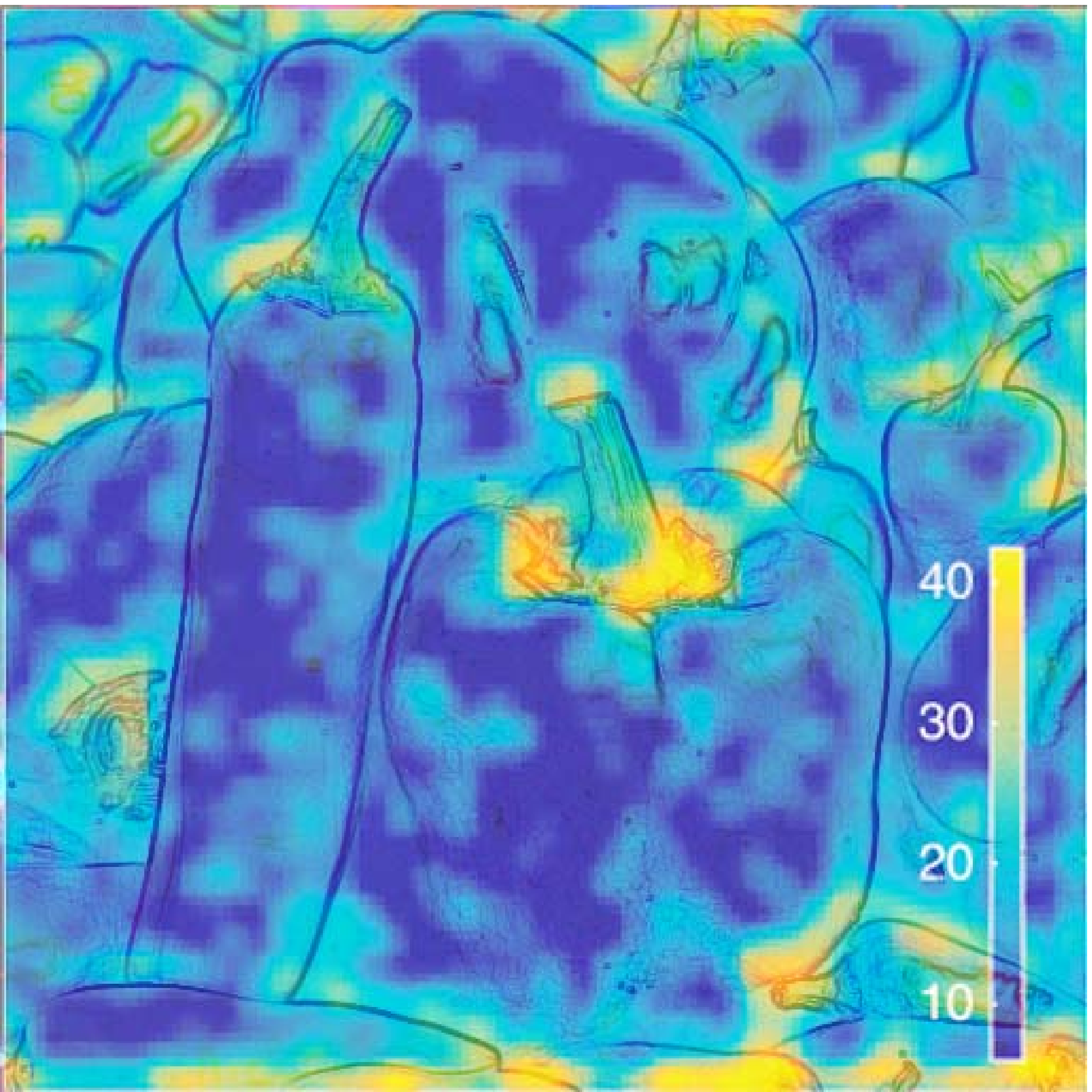}}
\caption{Visualization of the TT-rank maps in Example~\ref{ex_TT_denoising_image}.
Each entry of the map expresses the average of the sum of the TT-ranks of the TT-tensors which cover the corresponding pixel.}
\label{fig_tt_rank_maps}
\end{figure}

\begin{figure}[t!]
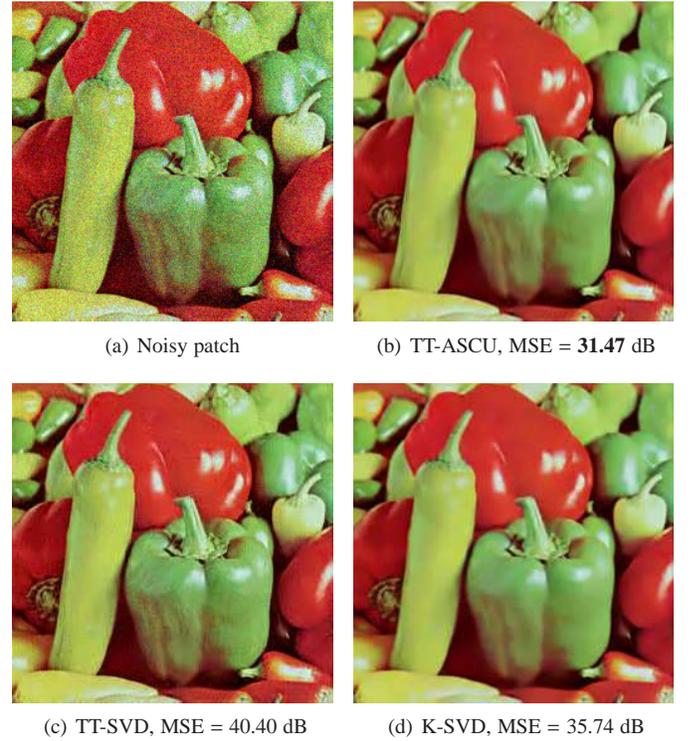

\centering
\subfigure[Noisy patch]{\includegraphics[width=.48\linewidth, trim = 0cm 0cm 0cm 0cm,clip=true]{fig_pepper_256_10dB}}
\hfill
\subfigure[TT-ASCU, MSE = {\bf 31.47} dB]{\includegraphics[width=.48\linewidth, trim =  0cm 0cm 0cm 0cm,clip=true]{fig_pepper_256_10dB_tt_ascu}
\label{fig_pepper_256_10dB_tt_ascu}}
\hfill
\subfigure[TT-SVD, MSE = 40.40 dB]{\includegraphics[width=.48\linewidth, trim =  0cm 0cm 0cm 0cm,clip=true]{fig_pepper_256_10dB_tt_truncation}
\label{fig_pepper_256_10dB_tt_truncation}}
\hfill
\subfigure[K-SVD, MSE = 35.74 dB]{\includegraphics[width=.48\linewidth, trim =  0cm 0cm 0cm 0cm,clip=true]{fig_pepper_256_10dB_ksvd}
\label{fig_pepper_256_10dB_ksvd}}
\caption{The Pepper image with added noise at 10 dB SNR, and the images reconstructed by different methods in Example~\ref{ex_TT_denoising_image}.}
\label{fig_pepper_10dB}
\end{figure}

\end{example}

\begin{example}{\bf Blind Source Separation of exponentially decaying signals from a single channel mixture.}\label{ex_TT_separation_exp}
In the final example, we considered a problem of blind source separation of three exponentially decaying signals from a single mixture $y(t)$ observed for a large number of samples $K =  3 \, 2^{16}$, given by 
\be
y(t) &=&   x_1(t)   + x_2(t)  + x_3(t) + e(t)		\notag 
\ee
where 
\be
x_r(t) = \exp({\frac{-5t}{rK}}) \, \sin(\frac{2\pi f_r}{f_s} t  + \frac{r\pi}{3}) \,,  \quad \text{for} \;\; r = 1, 2, 3,	\notag 
\ee
with $f_r  = 10, 10.1$ and 10.2 Hz, $f_s = 200$ Hz, and $e(t)$ represents an additive Gaussian noise.
The noisy mixture $y(t)$ at the signal-noise-ratio SNR = -10 dB, is plotted in Fig.~\ref{fig_tt_algs_expsin_sepa_R3_d15_y}.

\begin{figure}[t!]
\centering
\psfrag{Intensity}[cc][cc]{\scalebox{1}{\color[rgb]{0,0,0}\setlength{\tabcolsep}{0pt}\begin{tabular}{c}\smaller\smaller\smaller\smaller \vspace{1em}Intensity\end{tabular}}}%
\subfigure[Mixture at SNR =  -10 dB.]{\includegraphics[width=.47\linewidth, trim = 1cm 0cm 1cm 0cm,clip=true]{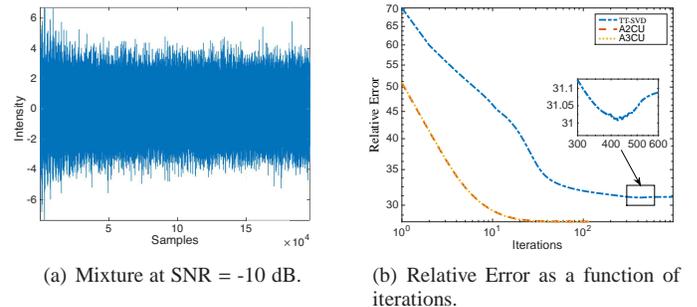}
\label{fig_tt_algs_expsin_sepa_R3_d15_y}}
\hfill
\psfrag{Approximation Error}[cc][cc]{\scalebox{1}{\color[rgb]{0,0,0}\setlength{\tabcolsep}{0pt}\begin{tabular}{c}\smaller\smaller\smaller\smaller \vspace{1em}Relative Error\end{tabular}}}%
\psfrag{Truncation}[lc][lc]{\scalebox{.6}{\color[rgb]{0,0,0}\setlength{\tabcolsep}{0pt}\begin{tabular}{l}\smaller\smaller\smaller\smaller\smaller   TT-SVD\end{tabular}}}%
\psfrag{ADCU}[lc][lc]{\scalebox{.6}{\color[rgb]{0,0,0}\setlength{\tabcolsep}{0pt}\begin{tabular}{l}\smaller\smaller\smaller\smaller\smaller   ADCU\end{tabular}}}
\psfrag{ATCU}[lc][lc]{\scalebox{.6}{\color[rgb]{0,0,0}\setlength{\tabcolsep}{0pt}\begin{tabular}{l}\smaller\smaller\smaller\smaller\smaller   ATCU\end{tabular}}}%
\subfigure[Relative Error as a function of iterations.]{\includegraphics[width=.45\linewidth, trim = 0.0cm 0cm 0cm 0cm,clip=true]{fig_tt_algs_expsin_sepa_R3_d15_SNR-10_x}
\label{fig_tt_algs_expsin_sepa_R3_d15_SNRm10}}
%
%
%
\caption{Illustration for Example~\ref{ex_TT_separation_exp} for signals of $K = 196,608$ samples and the signal-noise-ratio SNR = -10 dB.}
\label{fig_ex_TT_separation_exp}
\end{figure}

In order to separate the three signals $x_r(t)$ from the mixture $y(t)$,  we tensorized the mixture to an order-$16$ tensor $\tY$ of size $2 \times 2 \times \cdots \times 2 \times 6$. With this tensorization, each  decaying signal $x(t)$ had a TT-representation of rank-$(2,2,\ldots, 2)$. Hence, we were able to approximate $\tY$ as a sum of three TT-tensors $\tX_r$ of rank-$(2,2,\ldots, 2)$, that is
\be 
\min  \quad \|\tY - \tX_1 - \tX_2  - \tX_3\|_F^2 \, .\notag 
\ee
For this purpose, we sequentially fit a TT-tensor $\tX_r$ to the residual of the data $\tY$ with its approximation by the two other TT-tensors $\tX_s$ where $s \neq r$, that is, 
\be
\argmin_{\tX_r} \|\tY_{r} - \tX_r \|_F^2	\label{eq_fit_Yr} 
\ee
where $\tY_r =  \tY - \sum_{s \neq r} \tX_s$, for $r = 1, 2, 3$.

The TT-SVD algorithm applied to the above problem (\ref{eq_fit_Yr}) was not able to obtain satisfactory estimates of the three sources. The mean SAE of the estimated signals was only 8.07 dB.
{\emph{This is because after the first few iterations, the TT-SVD tended not to work well}}. The new estimates $\tX_r$, obtained by the TT-SVD, were therefore not always better than the previous estimates. As a consequence, the global cost function did not always decrease, as seen in Fig.~\ref{fig_tt_algs_expsin_sepa_R3_d15_SNRm10}. 

In contrast, when using the proposed ASCU algorithm, we obtained the three estimated signals with SAEs of 17.49, 14.17   and 15.70 dB, respectively. The algorithm converged after 120 iterations.

%


\end{example}

\section{Conclusions and Further Extensions}\label{sec::conclusions}

We have presented novel algorithms for the TT decomposition,
which are capable of adjusting ranks of two or three core tensors while keeping the other cores fixed.
Compared to the TT-SVD, the proposed algorithms have achieved lower approximation errors 
for the decomposition with a given TT-rank, and yielded tensors with lower TT-ranks for constrained approximations with a prescribed error tolerance.
By employing progressive computation of contracted tensors and prior compression, the proposed algorithms have been shown to exhibit low computational complexity. 
The proposed algorithms can be naturally extended to the TT-decomposition with nonnegativity constraints or decompositions of incomplete data. The alternating multicore update methods can also be applied to the tensor chain decomposition. 
In the sequel of this study, we illuminate the use of the proposed algorithms in  blind source separation, and for a conversion of a TT-tensor to a low-rank tensor in 
CPD. The proposed algorithms are implemented in the Matlab package TENSORBOX which is available online at: {\url{http://www.bsp.brain.riken.jp/~phan/tensorbox.php}}.

\appendices

\section{The Alternating Double-Cores Update (ADCU)}\label{alg::TT_twocore_update}

Following on Section~\ref{sec:ascu_1}, we consider the case $m = n+1$.  In order to update two cores $\tX_n$ and $\tX_{n+1}$, 
the error function in (\ref{eq_cost_Gnm}) can be rewritten in the following form 
\begin{align}
D &=   \|\tY\|_F^2  - \|\tT_{n,n+1}\|_F^2 + \|\tT_{n,n+1} - \tX_n \bullet \tX_{n+1}\|_F^2 \label{eq_cost_GnGnp1} \notag\\
&= \|\tY\|_F^2  - \|\tT_{n,n+1}\|_F^2 + \| [\tT_{n,n+1}]_{(1,2)} - [\tX_n]_{(1,2)}  \, [\tX_{n+1}]_{(1)} \|_F^2   \notag 
\end{align}
where $\tT_{n,n+1}$ is an order-4 tensor of size  $R_{n-1} \times I_n  \times I_{n+1} \times R_{n+1}$. 
The TT-decomposition now becomes a low-rank matrix factorisation of $[\tT_{n,n+1}]_{(1,2)}$, which can be computed 
through the truncated SVD of  $[\tT_{n,n+1}]_{(1,2)} \approx \bU_n \diag(\sigma_{n,1}, \ldots, \sigma_{n,R_{n}^{\star}}) \, \bV_n^T$.
The rank $R_{n}^{\star}$, if not given, is the smallest number of singular values determined as in (\ref{eq_rank_Rn}).
The new estimate of $\tX_n$ can be either $\bU_n$ or $\bU_n \diag(\bsigma)$, depending on the update procedure. 

The ADCU operates in the same way as the DMRG algorithm\cite{White1993}, but with different error tolerances $\varepsilon_n^2$ for 
the sub-problem.

\section{The Alternating Triple-Cores Updates (ATCU)}\label{alg::TT_threecore_update}

This Appendix derives an alternating algorithm which updates three consecutive core tensors. Similar to the ADCU, the aim is to reduce the number of computation of tensor contractions $\tT_{n:m}$. Moreover, we show that the algorithm indeed estimates only two cores. Hence, the computational cost of this algorithm is not higher than that of the double-cores update algorithm.
First, we rewrite the error function in (\ref{eq_cost_Gnm}) as follows 
\be 
D 
=  \|\tY\|_F^2  - \|\tT_{n:n+2}\|_F^2 + \|\tT_{n:n+2} - \tX_n \bullet \tX_{n+1} \bullet \tX_{n+2}\|_F^2 \notag  \label{eq_cost_GnGnp1Gnp2} 
\ee
where the contracted tensor $\tT_{n:n+2}$ is an order-5 tensor of size  $R_{n-1} \times I_n  \times I_{n+1} \times I_{n+2} \times R_{n+2}$.

The tensor $\tT_{n:n+2}$ is next reshaped into an order-3 tensor $\tZ_n$ of size $(R_{n-1}I_n) \times I_{n+1} \times (I_{n+2} R_{n+2} )$, i.e., by performing a mode-((1,2),3,(4,5)) unfolding   $\tZ_n = [\tT_{n:n+2}]_{(1,2),3,(4,5)}$. The above objective function is then given in the form of the Tucker-2 decomposition of $\tZ_n$, that is 
\be
D = \|\tY\|_F^2  - \|\tZ_n\|_F^2 + \| \tZ_n  -  \bU_n \bullet \tX_{n+1}  \bullet {\bV_n}  \|_F^2  \,\notag  \label{equ_Tucker2_cost_n}
\ee
where $\bU_n = [\tX_n]_{(1,2)}$ is of size $R_{n-1}I_n \times R_{n}$ and $\bV_n = [\tX_{n+2}]_{(1)}$ is of size $R_{n+1} \times  I_{n+2}R_{n+2}$.
This problem is solved using the Tucker-2 algorithm in Section~\ref{sec:tt_order3} in a few inner-iterations. When the tensor size $R_{n-1} I_n$ or $R_{n+2} I_{n+2}$ is significantly larger than the rank $R_{n}$ or $R_{n+1}$, the tensor $\tZ_n$ can be compressed to the size of $R_{n} \times I_{n+1} \times R_{n+1}$ prior to the Tucker-2 decomposition.

In the left-to-right update procedure, we need to orthogonalize $\tX_{n+1}$ and $\tX_{n+2}$ but not the core tensor $\tX_{n}$, because $\bU_n$ is orthogonal.
From the closed-form of $\tX_{n+1} = \bU_n^T \bullet \tZ_n \bullet \bV_n^T$, we have 
\be
\tX_{n+1}  \ltimesx_2 \tX_{n+1}  
&=&  \diag(\blambda_n)	\notag 
\ee 
where $\blambda_n$ are the $R_{n+1}$ largest eigenvalues of $ (\bU_n^T \bullet \tZ)  \, \ltimesx_2   \,  (\bU_n^T \bullet \tZ)$. This indicates that the left-orthogonalisation to $\tX_{n+1}$ simply scales frontal slices $\tX_{n+1}(:,:,r)$ by a factor of $1/\sqrt{\blambda_n(r)}$ where $r = 1, \ldots, R_{n+1}$.


\bibliographystyle{IEEEbib}
\bibliography{bibligraphy_thesis,BIBTENSORS2016}

\begin{thebibliography}{10}

\bibitem{Harshman}
R.A. Harshman,
\newblock ``Foundations of the {PARAFAC} procedure: Models and conditions for
  an explanatory multimodal factor analysis,''
\newblock {\em UCLA Working Papers in Phonetics}, vol. 16, pp. 1--84, 1970.

\bibitem{Carroll_Chang}
J.D. Carroll and J.J. Chang,
\newblock ``Analysis of individual differences in multidimensional scaling via
  an $n$-way generalization of {E}ckart--{Y}oung decomposition,''
\newblock {\em Psychometrika}, vol. 35, no. 3, pp. 283--319, 1970.

\bibitem{Tucker66}
L.R. Tucker,
\newblock ``Some mathematical notes on three-mode factor analysis,''
\newblock {\em Psychometrika}, vol. 31, pp. 279--311, 1966.

\bibitem{Lathauwer_HOOI}
L.~{De Lathauwer}, B.~De Moor, and J.~Vandewalle,
\newblock ``{On the best rank-1 and rank-(R1,R2,\ldots,RN) approximation of
  higher-order tensors},''
\newblock {\em SIAM Journal of Matrix Analysis and Applications}, vol. 21, no.
  4, pp. 1324--1342, 2000.

\bibitem{CEM:CEM1206}
R.~Bro, R.~A. Harshman, N.~D. Sidiropoulos, and M.~E. Lundy,
\newblock ``Modeling multi-way data with linearly dependent loadings,''
\newblock {\em Journal of Chemometrics}, vol. 23, no. 7-8, pp. 324--340, 2009.

\bibitem{Favier-deAlmeida14}
G.~Favier and A.~de~Almeida,
\newblock ``Overview of constrained {PARAFAC} models,''
\newblock {\em EURASIP Journal on Advances in Signal Processing}, vol. 2014,
  no. 1, pp. 1--25, 2014.

\bibitem{Lath-BCM}
L.~{De Lathauwer},
\newblock ``Decompositions of a higher-order tensor in block terms -- {P}art
  {I}: {L}emmas for partitioned matrices,''
\newblock {\em SIAM Journal of Matrix Analysis and Applications}, vol. 30, no.
  3, pp. 1022--1032, 2008,
\newblock Special Issue {on Tensor Decompositions and Applications}.

\bibitem{Phan_tensordeflation_alg}
A.-H. Phan, P.~Tichavsk{\'y}, and A.~Cichocki,
\newblock ``Tensor deflation for {CANDECOMP/PARAFAC}. {Part 1}: {Alternating
  Subspace Update Algorithm},''
\newblock {\em IEEE Transactions on Signal Processing}, vol. 63, no. 12, pp.
  5924--5938, 2015.

\bibitem{Phan_ranksplitting_full}
A.-H. Phan, P.~Tichavsk{\'y}, and A.~Cichocki,
\newblock ``Tensor deflation for {CANDECOMP/PARAFAC}. {Part} 3: {Rank
  splitting},''
\newblock {\em ArXiv e-prints}, 2015.

\bibitem{Kolda08}
T.G. Kolda and B.W. Bader,
\newblock ``Tensor decompositions and applications,''
\newblock {\em SIAM Review}, vol. 51, no. 3, pp. 455--500, September 2009.

\bibitem{NMF-book}
A.~Cichocki, R.~Zdunek, A.-H. Phan, and S.~Amari,
\newblock {\em Nonnegative Matrix and Tensor Factorizations: Applications to
  Exploratory Multi-way Data Analysis and Blind Source Separation},
\newblock Wiley, Chichester, 2009.

\bibitem{Cich-Lathp}
A.~Cichocki, D.~Mandic, C.~Caiafa, A-H. Phan, G.~Zhou, Q.~Zhao, and L.~De
  Lathauwer,
\newblock ``Tensor decompositions for signal processing applications: {F}rom
  two-way to multiway component analysis,''
\newblock {\em Signal Processing Magazine, IEEE}, vol. 32, no. 2, pp. 145--163,
  2015.

\bibitem{OseledetsTT09}
I.V. Oseledets and E.E. Tyrtyshnikov,
\newblock ``Breaking the curse of dimensionality, or how to use {SVD} in many
  dimensions,''
\newblock {\em SIAM J. Scientific Computing}, vol. 31, no. 5, pp. 3744--3759,
  2009.

\bibitem{Klumper91}
A.~Klumper, A.~Schadschneider, and J.~Zittartz,
\newblock ``Equivalence and solution of anisotropic spin-1 models and
  generalized t-j fermion models in one dimension,''
\newblock {\em Journal of Physics A: Mathematical and General}, vol. 24, no.
  16, pp. L955, 1991.

\bibitem{Vidal03}
G.~Vidal,
\newblock ``Efficient classical simulation of slightly entangled quantum
  computations,''
\newblock {\em Physical Review Letters}, vol. 91, no. 14, pp. 147902, 2003.

\bibitem{Holtz-TT-2012}
S.~Holtz, T.~Rohwedder, and R.~Schneider,
\newblock ``The alternating linear scheme for tensor optimization in the tensor
  train format,''
\newblock {\em SIAM J. Scientific Computing}, vol. 34, no. 2, 2012.

\bibitem{KressnerEIG2014}
D.~Kressner, M.~Steinlechner, and A.~Uschmajew,
\newblock ``Low-rank tensor methods with subspace correction for symmetric
  eigenvalue problems,''
\newblock {\em SIAM Journal on Scientific Computing}, vol. 36, no. 5, pp.
  A2346--A2368, 2014.

\bibitem{DaSilva2015131}
C.~Da Silva and F.~J. Herrmann,
\newblock ``Optimization on the hierarchical {Tucker} manifold Ð applications
  to tensor completion,''
\newblock {\em Linear Algebra and its Applications}, vol. 481, pp. 131 -- 173,
  2015.

\bibitem{Kressner2014}
D.~Kressner, M.~Steinlechner, and B.~Vandereycken,
\newblock ``Low-rank tensor completion by {Riemannian} optimization,''
\newblock {\em BIT Numerical Mathematics}, vol. 54, no. 2, pp. 447--468, 2014.

\bibitem{Rauhut2015}
Holger Rauhut, Reinhold Schneider, and {\v{Z}}eljka Stojanac,
\newblock {\em Tensor Completion in Hierarchical Tensor Representations}, pp.
  419--450,
\newblock Springer International Publishing, Cham, 2015.

\bibitem{journals/siamsc/GrasedyckKK15}
Lars Grasedyck, Melanie Kluge, and Sebastian KrŠmer,
\newblock ``Variants of alternating least squares tensor completion in the
  tensor train format.,''
\newblock {\em SIAM J. Scientific Computing}, vol. 37, no. 5, 2015.

\bibitem{OseledetsTT11}
I.V. Oseledets,
\newblock ``Tensor-train decomposition,''
\newblock {\em SIAM J. Scientific Computing}, vol. 33, no. 5, pp. 2295--2317,
  2011.

\bibitem{Phan2012-Kron}
A~H Phan, A~Cichocki, P~Tichavsky, D~Mandic, and K~Matsuoka,
\newblock ``On revealing replicating structures in multiway data: {A} novel
  tensor decomposition approach,''
\newblock in {\em Proc. 10th International Conf. LVA/ICA, Tel Aviv, March
  12-15}, pp. 297--305. Springer, 2012.

\bibitem{Kressner2014MC}
D.~Kressner and F.~Macedo,
\newblock ``Low-rank tensor methods for communicating {Markov} processes,''
\newblock in {\em Quantitative Evaluation of Systems}, pp. 25--40. Springer,
  2014.

\bibitem{White1993}
S.R. White,
\newblock ``Density-matrix algorithms for quantum renormalization groups,''
\newblock {\em Physical Review B}, vol. 48, no. 14, pp. 10345, 1993.

\bibitem{2013arXiv1301.6068D}
S.~V. {Dolgov} and D.~V. {Savostyanov},
\newblock ``{Alternating minimal energy methods for linear systems in higher
  dimensions. Part I: SPD systems},''
\newblock {\em ArXiv e-prints}, Jan. 2013.

\bibitem{oseledets2014tt}
I.V. Oseledets, S.V. Dolgov, V.A. Kazeev, D.~Savostyanov, O.~Lebedeva,
  P.~Zhlobich, T.~Mach, and L.~Song,
\newblock ``{TT-Toolbox},'' 2014,
\newblock https://github.com/oseledets/TT-Toolbox.

\bibitem{DBLP:journals/pami/ZhaoZC15}
Q.~Zhao, L.~Zhang, and A.~Cichocki,
\newblock ``Bayesian {CP} factorization of incomplete tensors with automatic
  rank determination,''
\newblock {\em {IEEE} Trans. Pattern Anal. Mach. Intell.}, vol. 37, no. 9, pp.
  1751--1763, 2015.

\bibitem{Aharon2006}
M.~Aharon, M.~Elad, and A.~Bruckstein,
\newblock ``{$K$-SVD: A}n algorithm for designing overcomplete dictionaries for
  sparse representation,''
\newblock {\em IEEE Transactions on Signal Processing}, vol. 54, no. 11, pp.
  4311--3322, 2006.

\end{thebibliography}

\end{document}